\documentclass[12pt]{amsart}
\usepackage[margin=1in]{geometry}
\usepackage{amsmath}
\usepackage{amssymb}
\usepackage[v2]{xy}

\usepackage{psfrag}
\usepackage{graphicx}

\newcommand{\wide}{\widetilde}

\newcommand{\N}{\mathbb{N}}
\newcommand{\G}{\Gamma}
\newcommand{\ra}{\rightarrow}
\newcommand{\dehnone}[2]{\delta^{(1)}_{#1}(#2)}
\newcommand{\dehntwo}[2]{\delta^{(2)}_{#1}(#2)}
\newcommand{\area}[2]{\mathrm{Area}_{#1}{(#2)}}
\newcommand{\vol}[1]{\mathrm{Vol}(#1)}
\newcommand{\logl}{\log_{3}}
\newcommand{\stable}[2]{\mathbf {a}_{#1 #2}} 
\newcommand{\finiteness}{\mathcal{F}}
\newcommand{\ddi}[3]{\Theta_{#1}^{#2}(#3)}
\newcommand{\ddii}[3]{\Delta_{#1}^{#2}(#3)}

\newcommand{\tspace}{\rule[-.25cm]{0cm}{.75cm}}
\newcommand{\first}{\xi}
\newcommand{\second}{\nu}
\newcommand{\topbound}[1]{L_{#1}}
\newcommand{\botbound}[1]{M_{#1}}
\newcommand{\deltabound}[1]{B_{#1}}

\theoremstyle{plain}
\newtheorem{theorem}{Theorem}[section]
\newtheorem{prop}[theorem]{Proposition}
\newtheorem{lemma}[theorem]{Lemma}

\theoremstyle{definition}
\newtheorem{definition}[theorem]{Definition}
\newtheorem{remark}[theorem]{Remark}
\newtheorem{example}[theorem]{Example}
\newtheorem{observation}[theorem]{Observation}
\numberwithin{equation}{section}
\numberwithin{figure}{section}

\begin{document}

\title{Super-exponential 2-dimensional Dehn functions}

\author[J.~Barnard]{Josh Barnard}
\address{Dept.\ of Mathematics \& Statistics\\
	University of South Alabama\\
	Mobile, AL 36688}
\email{jbarnard@jaguar1.usouthal.edu}

\author[N.~Brady]{Noel Brady}
\address{Dept.\ of Mathematics\\
        University of Oklahoma\\
	Norman, OK 73019}
\email{nbrady@math.ou.edu}

\author[P.~Dani]{Pallavi Dani}
\address{Dept.\ of Mathematics\\
        Louisiana State University \\
	Baton Rouge, LA 70803}
\email{pdani@math.lsu.edu}

\date{\today}

\begin{abstract}
We produce examples of groups of type $\mathcal{F}_3$ with 2-dimensional Dehn functions of the form $\exp^n(x)$ (a tower of exponentials of height $n$), where $n$ is any natural number.  
\end{abstract}

\maketitle
\footnotetext[1]{N. Brady was partially supported by NSF grant 
no.\ DMS-0505707}

\tableofcontents

\begin{sloppypar}

\section{Introduction}

Dehn functions have a long and rich history in 
group theory and topology. The germ of the notion of  Dehn 
function was expounded by Max Dehn in his solution to the 
word problem for the fundamental groups of closed hyperbolic 
surfaces in~\cite{dehn1, dehn2}. Gromov~\cite{gromov1, 
gromov2} further developed the notion of Dehn function 
as a filling invariant for a finitely presented group, and proposed  
the investigation of higher dimensional filling invariants. 

A function $\delta\!:\N\to\N$  is called the ($1$-dimensional)  Dehn 
function of a finite presentation if it is the minimal function 
with the following property. Every word of length at most 
$x$ in the generators representing the identity element 
of the group can be expressed as a product of at most 
$\delta(x)$ conjugates of relators and their inverses. Stated more 
geometrically, 
every loop of combinatorial length at most $x$ in the universal 
cover of a presentation $2$-complex for the group, can be 
null-homotoped using at most $\delta(x)$ $2$-cells. The 
adjective $1$-dimensional refers to the fact that the 
function $\delta(x)$ measures the area of efficient disk fillings of 
$1$-dimensional spheres.  It is customary to drop the 
adjective $1$-dimensional, and to simply talk about 
Dehn functions of finite presentations. It is known 
that, up to coarse Lipschitz equivalence, the Dehn function is a 
well-defined geometric invariant of a finitely presented group.

Dehn functions are intimately connected to the solution of 
the word problem in finitely presented groups. 
For example, a group has a solvable 
word problem if and only if its  Dehn function is bounded 
by a recursive function. In particular, the existence of groups with unsolvable word problem implies that there are groups with Dehn function not bounded above by any recursive function.
We now have a  greater 
understanding, thanks to the intense research activity of the past two 
decades and in particular to the deep work of~\cite{brs}, of 
which functions can be Dehn functions of 
finitely presented groups. 
For example, combining the results of~\cite{brbr} and~\cite{brs}, we know that the set of exponents $\alpha$ for 
which $x^{\alpha}$ is coarse Lipschitz equivalent to a 
Dehn  function is dense in $\{1\} \cup [2,\infty)$. 

Following Gromov, for each integer $k \geq 1$ one can define 
$k$-dimensional Dehn functions for groups 
$G$ of type ${\mathcal F}_{k+1}$ (that is, $G$ admits a $K(G,1)$ with finite $(k+1)$-skeleton). Roughly speaking, a  $k$-dimensional 
Dehn function, 
$\delta^{(k)}_{G}(x)$, gives a minimal upper bound for 
the number of $(k+1)$-cells needed in a $(k+1)$-ball 
filling of a singular $k$-sphere in $X$ of $k$-volume 
at most $x$. 
As in the case of ordinary Dehn functions, the coarse Lipschitz 
equivalence class of a $k$-dimensional 
Dehn function is an invariant of a group of type ${\mathcal F}_{k+1}$. 
The $k$-dimensional Dehn functions are 
examples of Gromov's higher dimensional filling invariants of 
groups.  Precise definitions of the $\delta^{(k)}_{G}(x)$ are given in 
Section~\ref{sec:background}.

In recent years, there has been considerable progress in our understanding 
of which functions are  ($k$-dimensional) Dehn functions 
of groups~\cite{alonso+, awp, bbfs, bf, brid, pride-wang}. 
Combining the results of~\cite{bbfs, bf},  one now knows that for each 
$k \geq 2$,  the set of exponents 
$\alpha$ for which $x^{\alpha}$ is coarse Lipschitz equivalent to 
a $k$-dimensional Dehn function is dense in $[1, \infty)$.  
Young~\cite{young} has shown that for $k\geq 3$ there are groups with $k$-dimensional Dehn function not bounded above by any recursive function. 
What Young actually shows is that for $k \geq 2$ 
there exist groups for which the Dehn functions, defined in terms 
of admissible $(k+1)$-dimensional manifold fillings of admissible $k$-tori, are  not bounded above 
by any recursive function. For $k \geq 3$,  he then 
uses the trick introduced in Remark 2.6(4) of \cite{bbfs} to conclude that 
the Dehn functions $\delta^{(k)}$ of these groups have the same property.  
In contrast,  Papasoglu~\cite{papa-recursive} shows that $2$-dimensional Dehn functions are all bounded above by recursive functions.    In the case $k=2$, Young's example together with 
Papasoglu's result show that the Dehn function based on admissible $3$-manifold fillings 
of $2$-tori is different from the Dehn function $\delta^{(2)}$ introduced above. 
In summary, Dehn functions based on admissible $(k+1)$-ball fillings of admissible 
$k$-spheres do not have to be bounded above by recursive functions for $k\not=2$. 
The case $k=2$ is indeed special. 

Papasoglu  
asked if there were examples of groups with super-exponential 
$2$-dimensional Dehn functions.  See also the remarks by Gromov in 5.D.(6) on page 100 
of \cite{gromov2}.  
Pride and Wang~\cite{pride-wang} produced an example of a group whose $2$-dimensional 
Dehn function lies between  $e^{\sqrt x}$ and $e^x$. The purpose 
of the current paper is to describe two families of groups $\{G_{n}
\}_{n=2}^{\infty}$ and $\{H_{n}\}_{n=2}^{\infty}$ whose $2$-dimensional Dehn functions display a range of super-exponential 
behavior.

\begin{theorem}\label{main-theorem}
For every $n>0$, there exist groups $H_n$ and $G_n$ of type $\finiteness_3$, with $\dehntwo{H_n}{x}\simeq \exp^n(\sqrt x)$
and $\dehntwo{G_n}{x}\simeq \exp ^n(x)$. 
\end{theorem}

Here $\exp^{n}$ denotes the $n$-fold composition of exponential 
functions. 

\begin{remark}\label{rem:nobors}
The combination-subdivision techniques of the current paper 
are one way of producing groups with super-exponential 
$2$-dimensional Dehn functions. There are other strategies 
that have the potential to produce groups 
of type ${\mathcal{F}}_{3}$ whose $2$-dimensional Dehn 
functions have super-exponential behavior. The following 
was suggested to the second author by Steve Pride. 

Take a group $A$ of type ${\mathcal{F}}_{3}$ with 
polynomially bounded ordinary Dehn function and containing a 
finitely presented subgroup $B$ whose 
ordinary  Dehn function is some tower of exponentials.  
Now let $G$ be the HNN extension with base group 
$A$, edge group $B$, and both edge maps are just the 
inclusion $B \to A$. Because of the finiteness 
properties of $A$ and $B$, we know that $G$ is 
of type ${\mathcal{F}}_{3}$. Finally take a word 
in $B$ that has very large area filling in $B$, and 
efficient filling in $A$. The 
product of this word with the HNN stable letter gives 
an annulus, whose ends  can be capped off using the efficient 
$A$-fillings. This gives a $2$-sphere in the HNN 
space with small area 
but large $3$-dimensional filling.  

So the whole problem reduces to finding groups 
with subgroups whose area is highly distorted. 
It is tempting to use the remarkable constructions 
of Birget-Ol$'$shanskii-Rips-Sapir~\cite{bors} to 
this end. However, it is likely that 
the ambient groups in~\cite{bors} are not of type 
${\mathcal{F}}_{3}$.  In particular, using the 
notation of~\cite{bors}, start with 
the Baumslag-Solitar group $BS(1,2)=\langle a,t\ |\ tat^{-1}=a^2\rangle$ as the 
group $G_b$, and then construct the group $H_N(S)$. The Cayley graph of $G_b$ is contained in that of $H_N(S)$. Consider the loop $[a,t^nat^{-n}]$, which is filled by a disk of the form shown in Figure 6 of~\cite{bors}. (This figure is viewed as a punctured sphere diagram, with the small loop labeled $u_b$ as the boundary, where the disk fillings track the behavior of the $S$-machine that reduces $[a,t^nat^{-n}]$ to the identity.) This disk filling has polynomial area. On the other hand, this loop is a product of exponentially many loops of the form $tat^{-1}a^{-2}$, each of which admits a disk filling of the form in Figure 6. These two methods of filling the word $[a,t^nat^{-n}]$ produce a $2$-sphere in the Cayley complex of $H_N(S)$, indicating the possibility that $H_N(S)$ is not of type $\finiteness_3$. Indeed, it is likely that the group $G_N(S)$  of \cite{bors} is not of type $\finiteness_3$ either. 
\end{remark}

This paper is organized as follows. In Section~\ref{sec:background} 
we give the definitions of $k$-dimensional Dehn functions, and 
remind the reader of a few standard techniques that are useful 
for establishing lower bounds. Section~\ref{sec:overview} gives 
a geometric overview of the groups, and the inductive 
construction of 
the ball-sphere pairs that are used to establish the lower bounds 
for the Dehn functions. Since the precise definitions of the groups 
are so involved, it is important that the reader keep the overview 
in mind when reading the later sections. In Section~\ref{sec:group-definitions} we give the precise definitions of the groups 
$H_{n}$ and $G_{n}$. These are defined recursively, and the details 
are recorded in Table~\ref{groups-table} for the reader's convenience.

The proof of Theorem~\ref{main-theorem} is broken into two parts. We establish the appropriate lower bounds in Section~\ref{sec:lower} 
by exhibiting sequences of embedded ball-sphere
pairs in the universal covers of $3$-dimensional $K(G,1)$-complexes. In Section~\ref{sec:upper} we obtain upper bounds on 2-dimensional Dehn functions of general graphs of groups in terms of information about their vertex and edge groups. These are then used to prove the desired upper bounds for $\delta_{H_n}^{(2)}$ and $\delta_{G_n}^{(2)}$. This proof depends on a key {\em area distortion} result, which is established in Section~\ref{sec:area-distortion}. 

We thank the referee for helpful comments and corrections.

\section{Preliminaries} \label{sec:background}

In this section we define $k$-dimensional Dehn functions, and 
describe two useful techniques for establishing lower bounds 
for Dehn functions. As in the case of ordinary (or $1$-dimensional) 
Dehn functions, one must work with coarse Lipschitz classes of 
functions in order to obtain a geometric invariant of a group. 

We begin with a discussion of  coarse Lipschitz equivalence. 
Given two functions $f,g\!:[0, \infty) \to [0, \infty)$ we define $f
\preceq g$ to mean that there exists a positive constant  $C$ such that 
\[ f(x) \leq C\,g(Cx) + Cx \]
for all $x\geq 0$. If $f \preceq g$ and $g \preceq f$ then $f$ and $g$
are said to be \emph{coarse Lipschitz equivalent} (or simply 
{\em equivalent}), denoted $f \simeq g$.

We work with the same definition  of high-dimensional 
Dehn  functions as in~\cite{bbfs}, which is equivalent to  
the notions in~\cite{brid} and  in~\cite{awp}.  
In order to give a formal definition of $\delta^{(k)}(x)$, one needs 
to work with a suitable class of maps that facilitates the counting 
of cells. 
In~\cite{bbfs} the class of {\em admissible maps} is used, and we 
use this class here. 

If $W$ is a compact $k$-dimensional manifold and $X$ a CW complex, an
\emph{admissible map} is a continuous map $f\!:W \to X^{(k)}\subset X$
such that $f^{-1}(X^{(k)} - X^{(k-1)})$ is a disjoint union of open
$k$-dimensional balls, each mapped by $f$ homeomorphically onto a
$k$-cell of $X$. 
The key fact about admissible maps is that every continuous map 
is homotopic to an admissible one.

\begin{lemma}[Lemma 2.3 of~\cite{bbfs}]\label{admissible} 
Let $W$ be a compact manifold (smooth or PL) of dimension $k$ and let $X$
be a CW complex. Then every continuous map $f\!:W \to X$ is homotopic 
to an admissible map. If $f(\partial W) \subset X^{(k-1)}$ then the homotopy
may be taken rel $\partial W$. 
\end{lemma}

Given an admissible  map $f\!:W \to X$ from a compact $k$-manifold 
to a CW complex $X$,  the $k$-\emph{volume} of $f$,
denoted $\text{Vol}_{k}(f)$, is defined to be the number of open 
$k$-balls in $W$ mapping to
$k$-cells of $X$. 
We now have all the ingredients for a formal 
definition of the $k$-dimensional Dehn function. 

\medskip

Given a group $G$ of type $\mathcal{F}_{k+1}$, fix an aspherical CW
complex $X$ with 
fundamental group $G$ and finite $(k+1)$-skeleton. Let $\widetilde{X}$ be
the universal cover of $X$. If $f\!:S^k \to \widetilde{X}$ is
an admissible map, define the \emph{filling volume} of $f$ to be the
minimal volume of an admissible extension of $f$ to $B^{k+1}$: 
\[ \, \text{FVol}(f) \ = \ \min \{\, \text{Vol}_{k+1}(g) \ | \  g\!:B^{k+1} \to
\widetilde{X}, \  g\vert_{\partial B^{k+1}} = f \, \}.\]  
 Note that 
extensions of $f$ exist since $\widetilde{X}$ is contractible, and 
that we may assume these extensions are admissible by Lemma \ref{admissible}.  

Now, the \emph{$k$-dimensional Dehn function} of
$X$ is defined to be 
\[\delta^{(k)}(x)\;=\;\sup \{ \, \text{FVol}(f) \ | \  f\!:S^k \to
\widetilde{X}, \ {\hbox{$f$ is admissible,}} \ \text{Vol}_k(f) \leq x \, \}.\]

This definition is the one given in~\cite{bbfs}, which is equivalent to the 
definitions given in~\cite{awp, brid}.    
It is shown 
in~\cite{awp} that, up to equivalence, $\delta^{(k)}(x)$ is 
a quasi-isometry invariant of $\widetilde{X}$.   In particular, the equivalence 
class of $\delta^{(k)}(x)$ does not depend 
on the particular $K(G,1)$ complex used. 
Therefore, we will often  write
$\delta^{(k)}(x)$ as $\delta^{(k)}_G(x)$.

\begin{remark}
We are concerned in this paper primarily with $2$- and $3$-dimensional volume, which for clarity we write as $\text{Vol}_2(f)=\text{Area}(f)$ and $\text{Vol}_3(f)=\text{Vol}(f)$. We also abuse notation in the standard way by writing, for example, $\text{Vol}(Y)$ to mean $\text{Vol}(f)$ for some understood $f\!:Y\to X$.
\end{remark}

\medskip

We recall two very useful techniques from \cite{bbfs} 
for  computing lower bounds for $k$-dimensional Dehn 
functions. 
 
\begin{remark}\label{rem:embedded}
The first technique follows from the fact (see Remark 2.7 of 
\cite{bbfs}) that  the volumes of top-dimensional embedded balls in 
a contractible complex are precisely the filling volumes of 
their boundary spheres. More precisely, 
if $\widetilde{X}$ is a contractible $(k+1)$-complex, and $g\!:B^{k+1} \to \widetilde{X}$ is an embedding, such that  $g$ and $g|_{S^{k}} = f$ are both 
admissible, then 
$\text{FVol}(f) = \text{Vol}_{k+1}(g)$. The proof is a standard 
homological argument; see~\cite{bbfs} for details. 
\end{remark}

\begin{remark}\label{rem:sparse}
See Remark 2.1 of~\cite{bbfs}. 
In order to establish the relation $f \preceq g$ between non-decreasing
functions, it suffices to consider relatively sparse sequences of
integers. For if $(n_i)$ is an unbounded sequence of integers for which
there is a constant $C>0$ such that $n_0 = 1$ and $n_{i+1} \leq Cn_i$ for
all $i$, and if $f(n_i) \leq g(n_i)$ for all $i$, then $f \preceq
g$. Indeed, given $x \in [0, \infty)$ there is an index $i$ such that
$n_i \leq x \leq n_{i+1}$, whence $f(x) \leq f(n_{i+1}) \leq g(n_{i+1})
\leq g(C n_i) \leq g(Cx)$. 
 \end{remark}

In the case of the groups $H_{j}$ in this paper, we produce 
a sequence of embedded balls $\{B^{3}_{r}\}_{r=1}^{\infty}$ and spheres $\{S^{2}_{r}\}_{r=1}^{\infty}$ in the universal cover of a 
$3$-dimensional,  aspherical $K(H_{j}, 1)$ complex. By Remark~\ref{rem:embedded} 
$\text{Vol}(B^{3}_{r})$ provides a lower bound 
for $\delta^{(2)}$ evaluated at $\text{Area}(S^{2}_{r})$. 
In Lemma~\ref{lem:Hvol} we prove that there exist 
constants $A_{1}, A_{2}$ so that 
$$
A_{1} r^{2} \ \leq \ \text{Area}(S^{2}_{r}) \ \leq \ A_{2} r^{2}\, .
$$
The sequence of numbers $\{\text{Area}(S^{2}_{r})\}_{r=1}^{\infty}$ 
satisfies the conditions of Remark~\ref{rem:sparse}, and so 
one obtains a lower bound for the function $\delta^{(2)}_{H_{j}}(x)$.  

The situation is similar for the groups $G_{j}$, although the inequalities are a bit more subtle. We show directly in Lemma~\ref{lem:Gvol} that the sequence $\{\text{Area}(S_r^2)\}_{r=1}^{\infty}$ satisfies the conditions of Remark~\ref{rem:sparse}, and so lower bounds for $\delta_{G_j}^{(2)}(x)$ evaulated at $\text{Area}(S_r^2)$ (given by the 3--volume of the corresponding embedded 3--balls), translate into a general lower bound for the function $\delta_{G_j}^{(2)}(x)$.

\section{Overview and geometric intuition}
\label{sec:overview}

In Section~\ref{sec:group-definitions} we define a sequence of groups 
$G_n$ and $H_n$ (for $n \geq 1$) with $\delta^{(2)}_{G_n}(x) \simeq  
\exp^n(x)$ and $\delta^{(2)}_{H_n}(x) \simeq  
\exp^{n}(\sqrt{x})$.  In Section~\ref{sec:lower} lower bounds for these Dehn functions are established by exhibiting sequences of 
embedded $3$-balls and boundary $2$-spheres in the universal 
covers of $3$-dimensional $K(G_n,1)$ and $K(H_n,1)$ complexes.  Some 
of the details in these two sections may appear a little daunting on first reading. 
In order to motivate these groups and the sequences of $3$-balls, we first 
investigate a classical construction for producing groups whose 
$1$-dimensional Dehn functions are iterated exponentials. 

\subsection{A $1$-dimensional Dehn function example and schematic diagrams}
Consider how one may construct 
finitely presented groups whose $1$-dimensional Dehn functions are 
iterated exponential functions. The following examples 
are part of the folklore, but only appear to have been written down 
relatively recently \cite{bridson}. Consider the family of finitely presented groups $BS_n$ (for each positive integer $n$), defined 
by taking $BS_1$ to be the Baumslag-Solitar group 
$$
BS_1 \; = \; \langle a, t_1 \, | \, t_1at_1^{-1} = a^2 \rangle
$$
and defining the $BS_n$ ($n \geq 2$) as iterated HNN extensions
$$
BS_{n+1} \; = \; \langle BS_n, t_{n+1} \, | \, t_{n+1}t_nt_{n+1}^{-1} = t_n^2 \rangle \, .
$$
The group $BS_n$ has Dehn function equivalent to the iterated exponential 
function $\exp^n(x)$.  

The intuition is as follows.  The commutator word 
$$
w_{1,m_1} \; = \; [t_1^{m_1} a t_1^{-m_1}, a]
$$ 
has linear 
length in $m_1$ and has area $\exp(m_1)$ in $BS_1$.

Now $BS_1 < BS_2$, and the extra relation in the presentation 
of $BS_2$ ensures that one can replace the segments 
$t_1^{m_1}$ in the commutator word above by words of the form $t_2^{m_2} t_1t_2^{-m_2}$ 
of length approximately $\log_2(m_1)$. Thus, the word 
$$
w_{2,m_2} \; = \; [(t_2^{m_2} t_1t_2^{-m_2})a(t_2^{m_2} t_1t_2^{-m_2})^{-1}, a]
$$
has length a linear function of $m_2$, but area $\exp^2(m_2)$ in $BS_2$. 

Continuing in this fashion, one  can keep replacing subwords  of the form 
$t_j^{m_j}$ by $t_{j+1}^{m_{j+1}}t_jt_{j+1}^{-m_{j+1}}$ where $m_{j+1}$ 
is approximately $\log(m_j)$. One obtains a word 
$w_{n,m_n}$ representing the identity in $BS_n$ whose length is  
a linear function of $m_n$ and whose area is $\exp^n(m_n)$. The 
lower bound on the area of the word $w_{n,m_n}$ is obtained by 
realizing $w_{n,m_n}$ as the boundary of an embedded van Kampen 
disk in the universal cover of the standard 
presentation 2-complex for $BS_n$. One verifies that the area of the 
embedded disk is $\exp^n(m_n)$. 
Since this universal cover is 
contractible, the standard homological argument from 
Remark~\ref{rem:embedded} enables one to 
conclude that any van Kampen filling disk for $w_{n,m_n}$ has at 
least as much area as the embedded one.

\begin{figure}[h]
\begin{center}
\psfrag{w1}{{\small $w_{1,m_1}$}}
\psfrag{w2}{{\small $w_{2,m_2}$}}
\psfrag{w3}{{\small $w_{3,m_3}$}}
\includegraphics[width=4.6in]{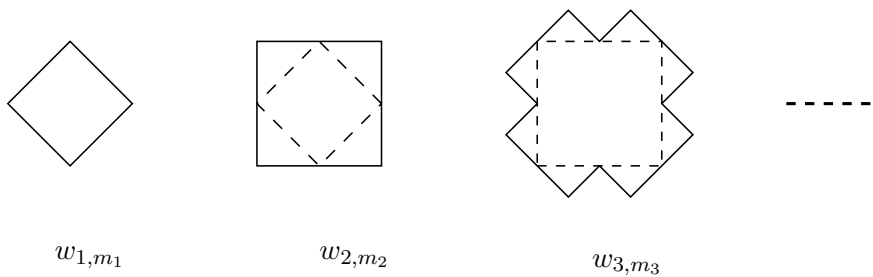}
\end{center}
\caption{Schematic van Kampen disks for  the words $w_{i,m_i}$. \label{fig:vk}}
\end{figure}

Figure~\ref{fig:vk} contains schematic diagrams of these embedded disks. 
In these schematics, we ignore all syllables  of length 1, so the original commutator 
$w_{1,m_1}$ appears as a square. Each replacement rule  
$$
t_j^{m_j} \; \; \mapsto \; \; t_{j+1}^{m_{j+1}}t_jt_{j+1}^{-m_{j+1}}
$$
is represented by a triangle with base attached along an edge 
of the previous schematic diagram.   The original square 
schematic for $w_{1,m_1}$ becomes the octagon schematic 
for $w_{2,m_2}$, and this becomes the $16$-gon schematic for   
$w_{3,m_3}$, and so on.

We reformulate the geometric intuition and combinatorics 
of these examples in terms of schematic diagrams. 
Start with a quadrilateral whose area is an exponential 
function of its  perimeter length. Now take logs of the perimeter length by attaching 
triangles via their base edges. Metrically, we think of the two 
remaining edges of each triangle as having length that is approximately a 
$\log$ of the length of the base edge. We attach a generation of 4 triangles, 
then a second generation of 8 triangles, a third generation of 16 triangles 
and so on. Repeat this procedure as often as 
necessary to obtain a diagram whose area is an $n$-fold composite of 
exponential functions of its perimeter length.

\begin{figure}[h]
\begin{center}
\psfrag{l}{{\scriptsize $\log(m)$}}
\psfrag{n}{{\small $m$}}
\includegraphics[width=120mm]{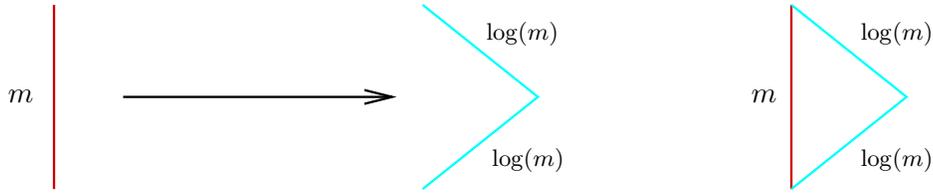}
\end{center}
\caption{The subdivision rule and associated 
triangular disk. \label{fig:comb1}}
\end{figure}

The combinatorics 
of the boundary circles of the schematic diagrams changes in a 
very simple fashion.  Edges get subdivided into two at each stage.  
Topologically this subdivision is achieved by attaching triangular 
$2$-disks along their base edges.  This is shown in Figure~\ref{fig:comb1}.

\begin{remark}
Note that for each $i$ the schematic diagram for $w_{i,m_i}$ is a 
\emph{template} for an infinite family of embedded van Kampen diagrams 
in the universal cover of the presentation $2$-complex for 
$BS_i$. The infinite family is obtained by letting the integer $m_i$ range from 
1 to infinity. 
\end{remark}

\subsection{The 2-dimensional Dehn function schematic procedure}
What  complications arise when one tries to mimic the procedure of 
the previous subsection in higher dimensions? We will focus 
on the case of $2$-dimensional Dehn functions. 
Start with a schematic $3$-dimensional ball and then try to reduce the areas of portions 
of its boundary $2$-sphere by attaching new $3$-balls along its boundary.    Each new 
3-ball attaches along a 2-disk. The  boundary of such a disc is a  circle that survives 
on the boundary of the new $3$-ball.  A \emph{collar neighborhood} of such a circle 
contributes to the area of the new boundary $2$-sphere.  

In order to be able to reduce area 
repeatedly by attaching successive generations of 
$3$-balls, we will need to ensure that future generations of 
$3$-balls attach in such a way as to cover up (large portions of) the boundary 
circles of the attaching disks of a given generation of $3$-balls.  This is a very basic 
complication in the combinatorics of attaching $3$-balls that does not occur 
in the case of 1-dimensional Dehn functions and the examples 
of the previous subsection. In the 
$1$-dimensional case 
the new $2$-disks attach along $1$-disks, the boundary of 
a $1$-disk is a $0$-sphere, and a collar neighborhood of a $0$-sphere 
makes no significant contributions to length of the new boundary circle.

We first describe the combinatorics of  our approach for dealing with the difficulty outlined in the preceding paragraph. We construct a sequence of schematic 
$3$-balls $B_i$ whose boundary $2$-spheres $S_i$ have the following properties. 
\begin{enumerate}
\item Each $S_i$ has a tiling into triangular regions. 
\item The  
 boundary edges of these regions have two colors, labeled  $(i-1)$ and $i$. 
 \item Each tile contains exactly one $(i-1)$-edge.  
 \end{enumerate}
 Now $S_{i+1}$ is obtained from $S_i$ by a \emph{combination-subdivision}
move as shown in Figure~\ref{fig:comb}.  Unlike the 
1-dimensional case, where a simple subdivision procedure was sufficient, 
in this case one combines  two  cells together to form a quadrilateral region, 
and then one subdivides that region. 

Each edge of color $(i-1)$ is the diagonal 
of a quadrilateral (combine two adjacent triangular regions).  
Delete this  diagonal edge and 
retriangulate this quadrilateral by adding a barycenter and coning 
to the four vertices.  Each new edge is of color $(i+1)$.  
Note that this gives a triangulation of the sphere $S_{i+1}$ with just two colors, 
$i$ and $(i+1)$, and now each triangle 
contains exactly one edge of color $i$.  Thus one can repeat this procedure 
inductively. 

\begin{figure}[h]
\begin{center}
\psfrag{i}{{\scriptsize $i$}}
\psfrag{m}{{\scriptsize $i-1$}}
\psfrag{p}{{\scriptsize $i+1$}}
\includegraphics[width=4in]{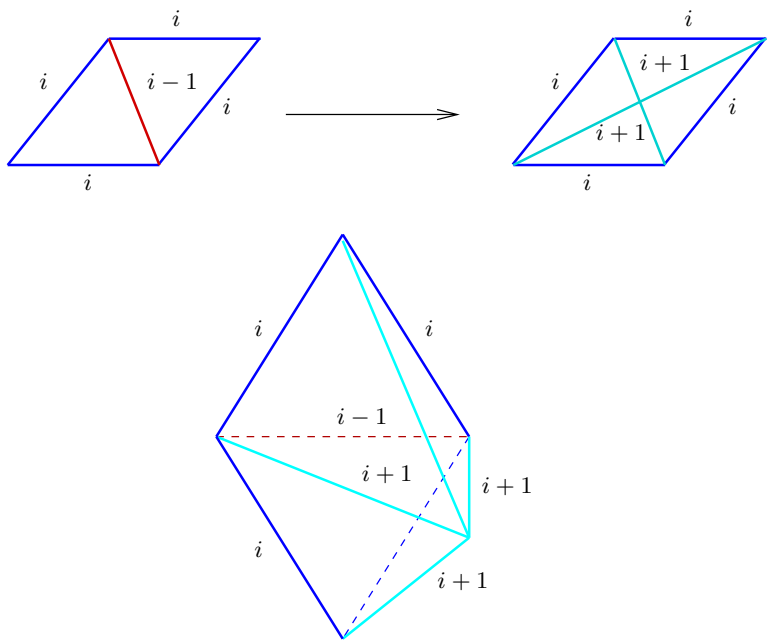}
\end{center}
\caption{The combination-subdivision rule and associated 
$3$-ball. \label{fig:comb}}
\end{figure}

We think of each sphere $S_i$ as being the boundary of a $3$-ball $B_i$. 
We think of each old quadrilateral as the back half of a $2$-sphere 
and each new (subdivided) quadrilateral as the front half of the  same 
$2$-sphere. This $2$-sphere bounds a $3$-ball, as indicated in the lower 
half of Figure~\ref{fig:comb}. The $3$-ball $B_{i+1}$ is 
obtained from $B_i$ by attaching a collection of such $3$-balls along the 
quadrilateral neighborhoods of the edges with color $(i-1)$.

\begin{figure}[h]
\begin{center}
\psfrag{i}{{\small I}}
\psfrag{j}{{\small II}}
\psfrag{n}{{\scriptsize $m$}}
\psfrag{l}{{\scriptsize $\log(m)$}}
\psfrag{e}{{\scriptsize $e^m$}}
\includegraphics[width=4.4in]{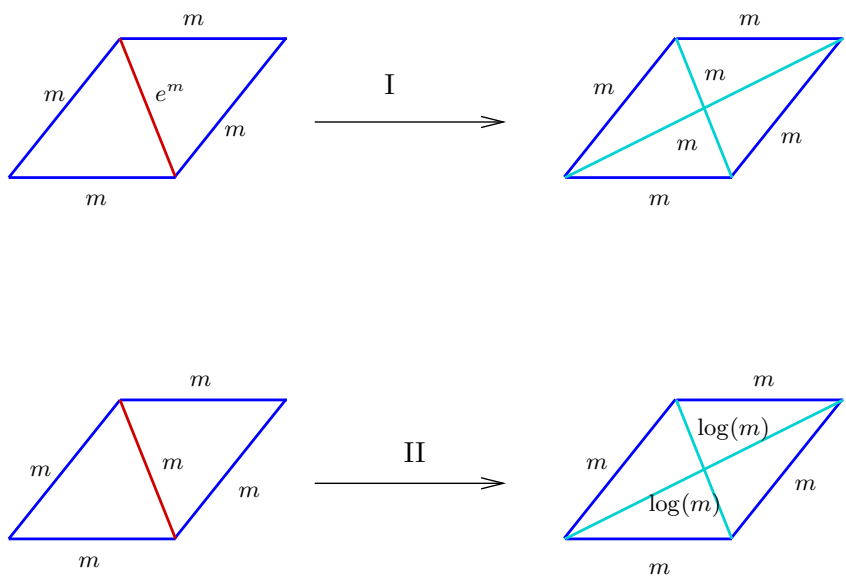}
\end{center}
\caption{Geometric versions of the combination-subdivision rule. \label{fig:metricsub}}
\end{figure}

As in the case of the $1$-dimensional Dehn function examples, the combination-subdivision 
moves have  geometric  content. Now we are concerned with areas as well as 
lengths. There are two geometric versions of the combination-subdivision rule. These are 
indicated in Figure~\ref{fig:metricsub}, and details are given in  Section~\ref{sec:lower}.  
\begin{enumerate}
\item In a Type I rule, edges of color $i$ 
have length  equal to $m$, edges of color $(i-1)$ have length $e^m$, 
and the new  edges of color $(i+1)$ have length 
$m$. Each of the two old triangles have area $e^m$ and the new triangles 
each have area $m^2$. 
\item In a Type II rule, edges of color $i$ and $(i-1)$ 
have length  equal to $m$, and the new edges of color $(i+1)$ have length 
$\log(m)$. Each of the two old triangles have area $m^2$ and the new triangles 
each have area of order $m$. 
 \end{enumerate}

\begin{figure}[h]
\begin{center}
\psfrag{i}{{\small I}}
\psfrag{j}{{\small II}}
\psfrag{1}{{\small $S_1$}}
\psfrag{2}{{\small $S_2$}}
\psfrag{3}{{\small $S_3$}}
\psfrag{4}{{\small $S_4$}}
\includegraphics[width=4in]{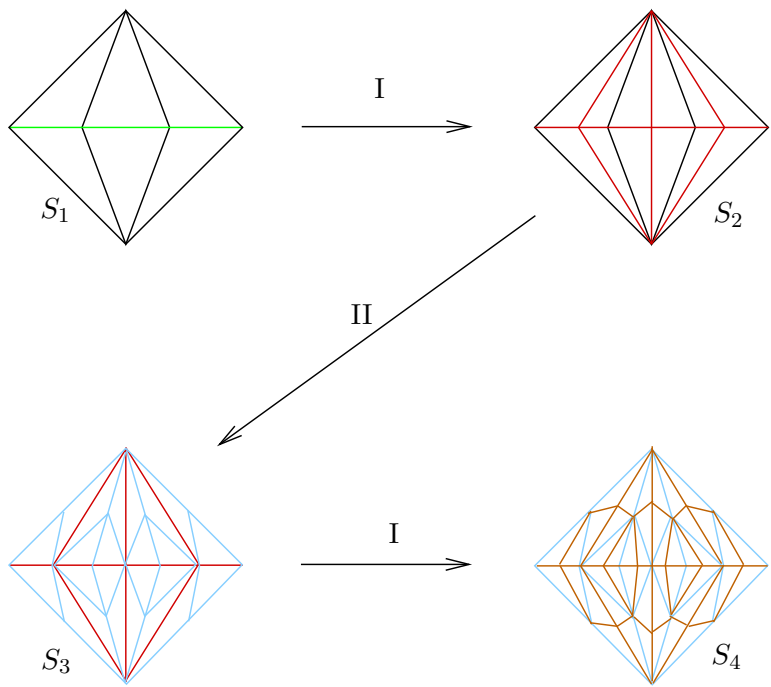}
\end{center}
\caption{Alternating the combination-subdivision rules. \label{fig:iterate}}
\end{figure}

These two types are designed so that one can alternate families of each type.  
In Section~\ref{sec:lower} we construct  families of embedded balls and boundary 
spheres with the following schematic diagrams. 
The base $2$-sphere $S_1$  is the join $S^0 \ast S^0 \ast S^0$, 
with the equator $S^0 \ast S^0$ being comprised of 4 edges of color $0$, and the 
remaining 8 edges having color $1$. We alternate these  versions of the geometric combination-subdivision rule as we progress 
along the sequence of schematic $2$-spheres. 
Figure~\ref{fig:iterate} shows three fourths of the 
sphere $S_1$, and the effect of the first few combination-subdivision rules. 

The geometry of these families is as follows. The sphere $S_1$ bounds a  $3$-ball  
whose volume is a quadratic function of the area of $S_1$. Applying  a generation 
of Type I moves, we see that the sphere $S_2$ bounds a 
$3$-ball whose volume is the square of $\exp(\sqrt{{\rm Area}(S_2)})$, which is 
equivalent to $\exp(\sqrt{{\rm Area}(S_2)})$. The sphere 
$S_3$ is obtained by applying Type II moves, and so 
bounds a $3$-ball whose volume is $\exp({\rm Area}(S_3))$, and so on.

\begin{remark}
These schematic diagrams are related to the groups $G_i$ and $H_i$ ($i \geq 0$) of 
Section~\ref{sec:group-definitions} as follows. 
For each integer $i\geq 0$, the sphere  $S_{2i+1}$ is a template for $2$-spheres in the universal 
cover of a $3$-dimensional $K(G_i,1)$-complex, 
and the sphere $S_{2i+2}$ is a template for $2$-spheres in the universal cover of a 
$3$-dimensional $K(H_i,1)$-complex. 
\end{remark}

\begin{remark}
In the case of the 1-dimensional Dehn function the schematic diagrams ignored 
syllables of length one. So a vertex may correspond to an edge of length 1 in a 
corresponding van Kampen diagram. 

Similarly, vertices in these schematic sphere diagrams $S_i$ may correspond 
to (several) 2-cells of the corresponding embedded spheres, and an edge 
may correspond to corridors on the corresponding spheres. 
\end{remark}

\begin{remark}\label{remark:F2}
By analogy with the $1$-dimensional case, one might be tempted in constructing the groups $G_i$ and $H_i$ to use Baumslag-Solitar type relations to provide the exponential scaling required by the Type I moves. For example, thinking of each of the triangles in the base quadrilateral of a Type I move as similar to the triangle in Figure~\ref{fig:comb1}, we could take the base quadrilateral to have boundary word $t^nat^{-n}s^na^{-1}s^{-n}$ in the group
$$
B\;=\langle a,t,s \,|\, a^s=a^t=a^2\rangle.
$$
This quadrilateral is composed of two triangles (each with area $\exp (n)$) with boundary words $t^nat^{-n}=a^{2^n}$ and $s^nas^{-n}=a^{2^n}$.
 
The Type I move is executed by adding a stable letter $u$ that acts on $B$ via the endomorphism $\varphi$ given by: $s\mapsto s, t\mapsto t, a\mapsto a^2$. The word $t^nat^{-n}s^na^{-1}s^{-n}$ is the boundary of a new quadrilateral, composed of four triangles (each of area $n^2$), and the new-color edges are $u^n$, $(s^{-1}u)^n$, $u^n$, and $(t^{-1}u)^n$. (See Figure~\ref{fig:comb}.)

Observe, however, that the group
$$
P\;=\langle B,u\, |\, g^u =\varphi(g), \;  (g \in B) \rangle
$$
is not an HNN extension with base $B$ since the map $\varphi$ is not injective. 
For instance $t^{-1}a t s^{-1}a^{-1}s \in \mathrm{Ker}(\varphi)$.

It is important that each of groups in the present paper 
admits a graph-of-groups decomposition. This facet of 
our construction is exploited heavily in Section~\ref{sec:lower} where the graph of groups 
structure is used to conclude the existence of $3$--dimensional classifying spaces, and these 
$3$--dimensional spaces are used to establish lower bounds for the $2$--dimensional 
Dehn functions. 

In Section~\ref{sec:group-definitions} below the graph-of-groups structure is guaranteed by using mapping tori of free groups of rank 2 in place of Baumslag-Solitar groups.   \end{remark}

\section{The groups}\label{sec:group-definitions}

Let $F_2$ denote the free group on two generators. 
To construct the two families of groups $G_n$ and $H_n$, we start with the base group $F_2 \times F_2 \times F_2$ and then alternate two procedures, coning and attaching suspended wings, which are described below. The base group is of type 
$\mathcal{F}_3$, and we shall see from the definitions of the two procedures that the groups obtained at each stage are also of type $\mathcal F_3$.

As indicated in Remark~\ref{remark:F2} we will need to use an exponentially growing automorphism $\varphi$ of $F_2$ repeatedly in the construction.  
It will be important (in Section~\ref{sec:moves}) that $\varphi$ be palindromic.
\begin{definition}[The palindromic automorphism $\varphi$] \label{def:phi} 
Setting $F_2=\langle \first, \second \rangle$, we define the automorphism $\varphi: F_2 \ra F_2$ by
\[
\varphi( \first)=\first \second \first \text{ and } \varphi(\second)=\first.
\]  
\end{definition}

\begin{remark}[Vector notation for free group bases] \label{rem:vector}
We use the vector notation ${\bf y}$ to denote the 
basis $\{y_1, y_2\}$ for a free group of rank two. In Table~\ref{groups-table}, the vectors themselves may have subscripts.  For example, ${\bf u}_1$ denotes the basis $\{u_{11}, u_{12}\}$ and $\stable 23$ denotes the basis $\{ a_{231}, a_{232}\}$. Furthermore, an ordered list of $k$ vectors describes an ordered 
basis of $F_{2k}$. For example, in Table~\ref{groups-table}, $\langle 
{\bf u}_0, {\bf y}\rangle$ denotes $F_4$ with the ordered basis 
$\{ u_{01}, u_{02}, y_1, y_2\}$. 

We also use product notation to denote coordinatewise multiplication of basis elements. For example,
${\bf u}_0^{-1}\stable 21$ denotes the basis $\{ u_{01}^{-1} a_{211}, u_{02}^{-1}a_{212}\}$.  

We say that two vectors ${\bf v}$ and ${\bf w}$ commute if the two basis elements represented by ${\bf v}$ commute with the two represented by ${\bf w}$.
\end{remark}

The following definition provides the algebraic framework for the Type II moves from Section~\ref{sec:overview}.  We shall see explicit examples of these moves 
in Section~\ref{somespheres}.

\begin{definition}[Coning] \label{def:coning} 
Suppose $G$ is a group with a subgroup $\G$ isomorphic to 
$(F_2*F_2*\cdots*F_2)\times F_2$. 
We define the \emph{cone of $G$ over $\G$}
to be the fundamental group of the graph of groups in Figure~\ref{fig:coneg}.

\begin{figure}[h]
\psfrag{G}[][][0.8]{$G$}
\psfrag{E}[][][0.8]{$\Gamma$}
\includegraphics[scale=0.6]{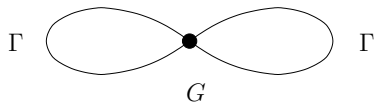}
\caption{The cone of $G$ over $\G$.\label{fig:coneg}}
\end{figure}
For each edge group, one of the  
edge maps is inclusion and the other edge map is 
$(\varphi*\varphi*\cdots*\varphi)\times\varphi$. This map is clearly injective, as $\varphi$ is injective.  Thus the cone of $G$ over $\G$ is simply a double HNN extension of $G$.  If $G$ is of type $\mathcal{F}_3$, then the cone of $G$ over $\G$ is also of 
type $\mathcal{F}_3$.
\end{definition}

\begin{example} \label{ex:G0}
In the first step of the inductive procedure, we start with
$G=F_2 \times F_2 \times F_2$, where the three $F_2$ factors are 
generated by the $2$-vectors $\stable 01$, $\stable 02$ and ${\bf y}$.  The subgroup $\G$ is
$F_2\times F_2= \langle \stable 01\rangle \times \langle \stable 02\rangle$.   Thus the presentation for 
$G_0=$ the cone of $G$ over $\G$ is given by
\[
G_0 =\langle G, {\bf u}_0 \mid (g,h)^{u_{01}}=(g,h)^{u_{02}}= (\varphi(g), \varphi(h)), \forall (g,h) \in \Gamma \rangle
\]
\end{example}

The next definition provides the algebraic framework for the Type I moves from Section~\ref{sec:overview}.  We shall see explicit examples of these moves 
in Section~\ref{somespheres}.

\begin{definition}[Attaching suspended wings]\label{def:suspend}
Let $G$ be a group containing a collection of subgroups $\{ \G_1,\ldots,\G_n \}$, each isomorphic to $F_2\rtimes_{\theta} F_4$, where the map $\theta\!:F_4 \to \mathrm{Aut}(F_2)$ is defined on the ordered basis $\{x_1, x_2, x_3, x_4\}$ as 
follows:
$$x_1 \mapsto \varphi, \quad
x_2 \mapsto \varphi, \quad
x_3 \mapsto \mathrm{id}, \quad x_4 \mapsto \mathrm{id}\,.
$$
Consider 
 the fundamental group of the graph of groups in Figure~\ref{fig:suspend}.
\begin{figure}[h]
\psfrag{G}[][][0.8]{$G$}
\psfrag{E1}[][][0.8]{$\Gamma_1$}
\psfrag{E2}[][][0.8]{$\Gamma_2$}
\psfrag{En}[][][0.8]{$\Gamma_n$}
\includegraphics[scale=0.8]{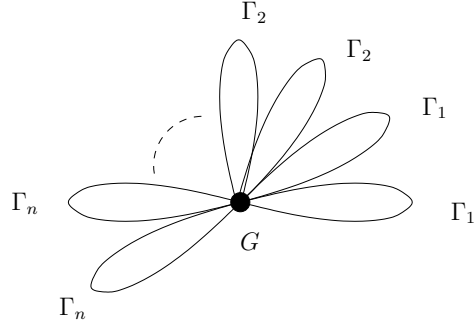}
\caption{Attaching suspended wings to the subgroups $\G_i$.\label{fig:suspend}}
\end{figure}
For each edge group $F_2 \rtimes_{\theta} F_4$, the two edge maps are the inclusion map and the map $\varphi \times \mathrm{id}$, which acts by $\varphi$ on the $F_2$ factor and the identity on the $F_4$ factor.  It is easy to check that the latter map is an injective homomorphism.
We say that this group 
is obtained from $G$ by \emph{attaching suspended wings} to the collection of subgroups $\G_i$.

Note that the edge groups $\G_i$ are all of type $\mathcal{F}_2$.  Thus, if 
 $G$ is of type $\mathcal{F}_3$, then the new group is also 
 of type $\mathcal{F}_3$.  
\end{definition}

\begin{example}\label{ex:susp-wings1} 
In the group $G_0$ above, the subgroup $\G_1=\langle \stable 01,{\bf u}_0, {\bf y}\rangle$ is isomorphic to the group $F_2 \rtimes_{\theta} F_4$ from 
Definition~\ref{def:suspend}.
Then the group 
obtained from $G_0$ by attaching suspended wings to $\G_1$ 
has the presentation
\[
\left\langle \, G_0, \stable 11 \;\bigg|\; \begin{array}{ll}g^{a_{111}}=g^{a_{112}}=\varphi(g) & \forall g \in F(\stable 01),\\ g^{a_{111}}=g^{a_{112}}=g & \forall g \in F({\bf u}_0) \ast F({\bf y})\end{array}\,
\right\rangle
\]
\end{example}

\begin{example}\label{ex:H1}
The group $H_1$ in the inductive construction is obtained from $G_0$ by attaching suspended wings to 
$\G_1$ and $\G_2$, where 
$\G_1= F_2 \rtimes_{\theta} F_4= \langle \stable 01 \rangle \rtimes_{\theta} \langle {\bf u}_0 , {\bf y} \rangle$ 
as in Example~\ref{ex:susp-wings1}, and 
$\G_2=F_2 \rtimes_{\theta} F_4= \langle \stable 02 \rangle \rtimes_{\theta} \langle {\bf u}_0 , {\bf y} \rangle$. The pair of stable letters in $H_1$ corresponding to the edge group $\G_1$ is denoted by $\stable 11$ and the pair of stable letters corresponding to the edge group $\G_2$ is denoted by $\stable 12$. 
This is the summarized in the $H_1$-row of Table~\ref{groups-table}.
\end{example}

\subsection{Inductive definition of the groups $G_n$ and $H_n$}\label{sec:inductive-group-defs}
In Example~\ref{ex:G0}, the group $G_0$ was defined 
by applying the coning procedure to the base group $G= F_2 \times F_2 \times F_2$. In Example~\ref{ex:H1}
the group $H_1$ was
obtained from $G_0$ by attaching suspended wings.  The groups $G_n$, $n\geq 1$ and $H_n$, $n \geq 2$ are defined inductively
by alternating these two procedures. 

The group $G_{n}$ is obtained from $H_{n}$ by coning over a subgroup of $H_{n}$
that is isomorphic to $F_{2^{n+1}} \times F_2$, where $F_{2^{n+1}}$ is a free product of $2^n$ copies of $F_2$.  The pair of new stable letters is denoted by ${\bf u}_n$.   The $F_2$ factor of the subgroup
$F_{2^{n+1}} \times F_2$ is generated by ${\bf u}_{n-1}$.  The generators of the 
$F_{2^{n+1}}$ factor are listed in the $G_n$-row of Table~\ref{groups-table}.

The group $H_n$ is obtained from
$G_{n-1}$ by attaching suspended wings to a collection of $2^{n}$
subgroups of $G_{n-1}$, each of which is isomorphic to $F_2\rtimes_{\theta} F_4$ from Definition~\ref{def:suspend}.  
Note that each of these subgroups labels two edges of the rose in Figure~\ref{fig:suspend}, so
that there are actually $2^{n+1}$ new stable letters in $H_n$. 
These are labeled by $2^n$ vectors
in the $H_n$-row of Table~\ref{groups-table}.  

Lemma~\ref{lem:edge-str} establishes that the edge groups listed in Column $3$ of Table~\ref{groups-table} are indeed isomorphic to 
$F_{2^i} \times F_2$ or $F_2 \rtimes_{\theta} F_4$, depending on the case. 
By induction, the groups $G_n$ and $H_n$ are of type $\mathcal{F}_3$ for all $n$. 
Table~\ref{groups-table} also includes the $2$-dimensional Dehn functions of the resulting groups.

{\scriptsize

\begin{table}
\caption{\label{groups-table}}
\begin{center}
\begin{tabular}{|c|c|c|c|}
\hline
\tspace
Group & Stable letters & Edge groups & $\dehntwo{}{s}$ \\
\hline

\tspace
$H_0$ & $\stable 01, \stable 02,  \mathbf y$ & & $x^{3/2}$\\
\hline

\tspace
$G_0$ & $\mathbf u_0$ & $\langle \stable 01 \rangle \times \langle \stable 02 \rangle$ & $x^2$\\
\hline

\tspace
$H_1$ 
& 
$
\begin{array}{c}
\tspace
\stable 11 \\
\stable 12
\end{array}
$ 
& 
$
\begin{array}{c}
\tspace
\langle \stable 01 \rangle \rtimes_{\theta} 
\langle \mathbf u_0, 
\mathbf y \rangle \\ 
\tspace
\langle \stable 02 \rangle \rtimes_{\theta} 
\langle \mathbf u_0, 
\mathbf y\rangle 
\end{array}
$
& 
$e^{\sqrt x}$\\
\hline

\tspace
$G_1$ & $\mathbf u_1$ & 
$\langle \stable 11, \stable 12 \rangle \times \langle \mathbf u_0\rangle$ & $e^x$\\
\hline

\tspace
$H_2$ & 
$\begin{array}{c} 
\tspace
\stable 21 \\
\tspace
\stable 22 \\
\tspace
\stable 23 \\
\tspace
\stable 24
\end{array}$
& 
$\begin{array}{c} 
\tspace
\langle \stable 11 \rangle \rtimes_{\theta} \langle \mathbf u_1, \mathbf y\rangle \\
\tspace
\langle \stable 12 \rangle \rtimes_{\theta} \langle \mathbf u_1,\mathbf y\rangle \\
\tspace
\langle \mathbf u_0^{-1} {\stable 11} \rangle \rtimes_{\theta} \langle \mathbf u_1, \stable 01\rangle\\ 
\tspace
\langle \mathbf u_0^{-1}{\stable 12}  \rangle \rtimes_{\theta} \langle \mathbf u_1, \stable 02 \rangle
\end{array}$   
& $e^{e^{\sqrt x}}$\\
\hline

\tspace
$G_2$ & $\mathbf u_2$ & $
\langle
\stable 21, \stable 22, \stable 23, \stable 24 
\rangle \times
\langle \mathbf u_1\rangle$ 
&  $e^{e^x}$\\
\hline

\tspace
$H_3$
& 
$\begin{array}{c}
\tspace
\stable 31\\ \tspace
\stable 32\\ \tspace
\stable 33 \\ \tspace
\stable 34 \\ \tspace
\stable 35 \\ \tspace
\stable 36 \\ \tspace
\stable 37 \\ \tspace
\stable 38
\end{array}$
& 
$\begin{array}{c}
\tspace
\langle \stable 21 \rangle \rtimes_{\theta} \langle \mathbf u_2, \mathbf y \rangle \\
\tspace 
\langle \stable 22 \rangle \rtimes_{\theta} \langle \mathbf u_2 , \mathbf y \rangle \\
\tspace 
\langle \stable 23 \rangle \rtimes_{\theta} \langle \mathbf u_2, \stable 01 \rangle \\
\tspace
\langle \stable 24 \rangle \rtimes_{\theta} \langle \mathbf u_2, \stable 02\rangle \\
\tspace
\langle
\mathbf u_1^{-1} \stable 21  \rangle \rtimes_{\theta} \langle \mathbf u_2, \stable 11  \rangle\\
\tspace
\tspace
\langle
\mathbf u_1^{-1}\stable 22 \rangle \rtimes_{\theta} \langle \mathbf u_2, \stable 12 \rangle\\
\tspace
\langle
\mathbf u_1^{-1}\stable 23 \rangle \rtimes_{\theta} \langle \mathbf u_2, \mathbf u_0^{-1}\stable 11  \rangle\\
\tspace
\langle
\mathbf u_1^{-1}\stable 24 \rangle \rtimes_{\theta} \langle \mathbf u_2, \mathbf u_0^{-1}\stable 12  \rangle\\
\end{array}$
&
$e^{e^{e^{\sqrt x}}}$\\ 
\hline

\tspace
\vdots & \vdots &\vdots &\vdots \\
\hline
$
\begin{array}{c}
\tspace
H_n \\
\tspace
(n\geq 1) 
\end{array}
$ &
$\begin{array}{c}
\tspace
\stable ni \\
\tspace
1 \leq i \leq 2^{n-1} \\
\tspace
\stable ni \\
\tspace
2^{n-1} <  i \leq 2^{n}
\end{array}$ 
& 
$\begin{array}{c}
\tspace
\langle \stable {(n-1)}i \rangle \rtimes_{\theta} \langle \mathbf u_{n-1}, \mathcal{L}_n(i) \rangle\\
\tspace 
1 \leq i \leq 2^{n-1} \\
\tspace
\langle \mathbf u_{n-2}^{-1} \stable {(n-1)}j \rangle \rtimes_{\theta} \langle \mathbf u_{n-1}, \mathcal{L}_n(j+2^{n-1}) \rangle  \\
\tspace
1 \leq j \leq 2^{n-1}, \;j=i-2^{n-1}
\tspace
\end{array}
$
& $\exp^{n}{\sqrt x}$\\
\hline
$
\begin{array}{c}
\tspace
G_n \\
\tspace
(n\geq 1) 
\end{array}
$
& $\mathbf u_n$ & 
$
\langle
\stable nj \rangle_{j=1}^{2^n} \times \langle \mathbf u_{n-1} \rangle$
& $\exp^n x$\\
\hline
\end{tabular}
\end{center}
\end{table}

}

\medskip

\subsection {Further details about Table~\ref{groups-table}.}  
This subsection provides a precise description of the edge groups 
$F_2\rtimes_{\theta} F_4$ in the definition of $H_n$.  
On a first reading the reader may wish to focus on the first few groups in 
Table~\ref{groups-table} (up to $G_2$).
The arguments for the lower bounds for the first few groups 
(Section~\ref{somespheres}) and for the upper bounds in general (Sections~\ref{sec:upper} 
and~\ref{sec:area-distortion}) can be worked through without a thorough knowledge of the general labeling.

Recall that $H_n$ is obtained from $G_n$ by attaching suspended wings 
along a collection of $2^n$ subgroups of the form $F_2\rtimes_{\theta} F_4$.  
These are labeled as follows:
$$
\langle \stable {(n-1)}i \rangle \rtimes_{\theta} \langle \mathbf u_{n-1}, \mathcal{L}_n(i) \rangle, \quad \text{ with } 
1 \leq i \leq 2^{n-1} \\
$$ and $$
\langle \mathbf u_{n-2}^{-1} \stable {(n-1)}j \rangle \rtimes_{\theta} \langle \mathbf u_{n-1}, \mathcal{L}_n(j+2^{n-1}) \rangle, \quad  \text{ with }
1 \leq j \leq 2^{n-1}\,. 
$$
Lemma~\ref{lem:edge-str} establishes that these groups are indeed isomorphic to $F_2\rtimes_{\theta} F_4$.  
Note that the generators of the $F_2$ factors are either the stable letters of $H_{n-1}$ or are ${\bf u}_{n-2}^{-1}$ times these stable letters. Note also that 
two of the generators of the $F_4$ factor are always ${\bf u}_{n-1}$, the 
stable letters in the definition of $G_{n-1}$. The remaining generators of the $F_4$ factor are described using the ordered list $\mathcal{L}_n$, which is defined recursively as follows:
\begin{align*}
&\mathcal{L}_1= \{ {\bf y, \; y} \} \\
&\mathcal{L}_2= \{ {\bf y,  \; y,} \; \stable 01, \; \stable 02 \}\\
&\mathcal{L}_n = \{ \mathcal{L}_{n-1},  \;\stable {(n-2)}1,\ldots,\stable {(n-2)}{2^{n-2}}, \; \mathbf u_{n-3}^{-1}\stable {(n-2)}1,\ldots, \mathbf u_{n-3}^{-1}\stable{(n-2)}{2^{n-2}}\},
\end{align*}
where this last equation is for $n \geq 3$. We let $\mathcal{L}_n(i)$ denote the $i$th element of $\mathcal{L}_n$. For $n= 1, 2$, and $3$ these labels are given explicitly in Table~\ref{groups-table}.

We now establish a few facts that will be 
used repeatedly in the proof of Lemma~\ref{lem:edge-str}.
Let $M(G, {\bf t})$ denote a multiple HNN-extension with base group $G$ and 
stable letters represented by the vector ${\bf t}$ (which may have more than two coordinates).

\begin{observation} \label{obs:freeproduct}
If $H$ is a subgroup of $G$ such that the intersection of $H$ with any of the edge groups of $M(G, {\bf t})$ is trivial, then 
the subgroup $\langle H, {\bf t} \rangle$ inside $M(G, {\bf t})$ is isomorphic to $H \ast \langle {\bf t}\rangle$.
\end{observation}

In the following lemma, we use the notation of Remark~\ref{rem:vector}.

\begin{lemma}\label{lem:trivialintersection}
Let $M(M(G,{\bf t}), {\bf s})$ be an iterated multiple HNN-extension.  Let ${\bf b}$ and ${\bf c}$ be vectors consisting of elements of $G$ with the following properties.  
\begin{enumerate}
\item The intersection of $\langle {\bf c} \rangle$ with each edge group of $M(M(G,{\bf t}),{\bf s})$ is trivial.
\item The vector ${\bf b}$ has the same number of coordinates as ${\bf t}$. 
\end{enumerate}
Then $\langle {\bf b}{\bf t} \rangle$ is a free group whose rank equals the number of coordinates of ${\bf t}$ and 
$$\langle {\bf b}{\bf t} \rangle \cap \langle 
{\bf s} , {\bf c}\rangle = 1. $$
\end{lemma}

\begin{proof}
There is a retraction $\rho$ of $M(G,{\bf t})$ onto the free group $F$ on ${\bf t}$, such that $\rho(g)=1$ for all $g \in G$.  The subgroup $\langle {\bf bt} \rangle$ of $M(G,{\bf t})$ 
has the same number of generators as $F$ and maps onto $F$ under 
this retraction. It is therefore free. 

Now suppose $w_1= w_1({\bf s}, {\bf c})$ 
and $w_2=w_2({\bf b t})$ are words that represent the same element 
of $M(M(G,{\bf t}),{\bf s})$. 
Note that $w_1$ is contained in $\langle   {\bf s}, {\bf c} \rangle$, which is isomorphic to 
$\langle {\bf s} \rangle * \langle {\bf c} \rangle$, by Condition (i) above and Observation~\ref{obs:freeproduct}.  
Moreover, $w_1$ represents an element of the base group for $M(M(G,{\bf t}), {\bf s})$, since $w_2$ does.   These two facts imply that $w_1$ can be freely reduced to a word that does not involve any coordinates of ${\bf s}$, say to $w_3 =w_3({\bf c})$.  Note that $w_3$ actually represents an element of $G$, so that $\rho(w_3)=1$.  This implies that $\rho(w_2)=1$.  Since $w_2$ is a word in ${\bf bt}$, this can only happen if $w_2$ is trivial. 
\end{proof}

\begin{remark}\label{rem:stablefree}
The subgroup $\langle {\bf t} \rangle$ of $M(G,{\bf t})$ is free.  
Since $M(G,{\bf t}) \to M(M(G, {\bf t}), {\bf s})$ is an injection, the stable letters ${\bf t}$ also generate a free group in $M(M(G, {\bf t}),{\bf s})$.  This will be used repeatedly in the proof of Lemma~\ref{lem:edge-str}.    
\end{remark}

\begin{lemma}[Structure of edge groups]\label{lem:edge-str}
The subgroup $\langle \stable 0i, {\bf u}_0, {\bf y} \rangle$  (where $i=1$ or $2$) of $G_0$ is isomorphic to $F_2 \rtimes_\theta F_4$.  Further, for each $n\geq 1$, the following statements hold.  

\noindent
{Coning:}
\begin{enumerate}
\item[(1)] The subgroup $\langle \stable n1, \ldots, \stable n{2^n}, {\bf u}_{n-1}\rangle$
of $H_n$ is isomorphic to $F_{2^{n+1}} \times F_2$, where the $F_2$ factor has basis ${\bf u}_{n-1}$.
\end{enumerate}

\noindent
{Suspended wings:}
\begin{enumerate}
\item[(2)] For $1 \leq i \leq 2^{n}$ the subgroup  
$\langle \stable {n}i, \mathbf u_{n}, \mathcal{L}_{n+1}(i) \rangle$ of $G_{n}$ is isomorphic to $F_2 \rtimes_{\theta} F_4$, where $F_2$ has basis $\stable {n}i $ and $F_4$ has ordered basis $\{\mathbf u_{n}, \mathcal{L}_{n+1}(i) \}$.
\item[(3)] 
For $1 \leq j \leq 2^{n}$, the subgroup 
$\langle \mathbf u_{n-1}^{-1} \stable {n}j, \mathbf u_{n}, \mathcal{L}_{n+1}(j+2^{n}) \rangle$ of $G_{n}$ is isomorphic to 
$F_2 \rtimes_{\theta} F_4$, where $F_2$ has basis $\mathbf u_{n-1}^{-1} \stable {n}j$ and $F_4$ has ordered basis $\{\mathbf u_{n}, \mathcal{L}_{n+1}(j+2^{n}) \}$.

\end{enumerate}
\end{lemma}

\begin{proof}
The proof is  by induction on $n$.  

For the base case, (the subgroup $\langle \stable 0i, {\bf u}_0, {\bf y} \rangle$ of $G_0$ 
with $i=1$ or $2$)
note that the subgroups $\langle \stable 0i \rangle$, and $\langle {\bf y} \rangle$ 
are free factors of the base group $(F_2)^3$ (see Table~\ref{groups-table}).  By 
Remark~\ref{rem:stablefree}, 
$\langle {\bf u}_0 \rangle \simeq F_2$ since it is generated by the stable letters of $G_0$.  Now $\langle {\bf u_0}, {\bf y} \rangle$ is isomorphic to $F_4$ by Observation~\ref{obs:freeproduct}.  The fact that $\langle {\bf u_0}, {\bf y} \rangle$ acts 
on $\langle \stable 0i \rangle$ via $\theta$ is obvious by construction.  Note that 
$G_0$ can be thought of as a single multiple HNN extension $M(F_2 \times F_2, \bf y,  {\bf u}_0)$.
Then $\langle {\bf u_0}, {\bf y} \rangle$, consisting of stable letters, intersects the subgroup 
$\langle  \stable 0i\rangle$ of the edge group trivially.  
This shows that $\langle \stable 0i, {\bf u_0}, {\bf y} \rangle$ is isomorphic to $F_2 \rtimes_{\theta} F_4$.

{\it Inductive step for (1):} 
The group $H_n$ is a multiple HNN extension $M( G_{n-1}, \stable n1, \ldots, \stable n{2^n} )$.  So by Remark~\ref{rem:stablefree}, the subgroup 
  $\langle \stable n1, \ldots, \stable n{2^n} \rangle$ is isomorphic to $F_{2^{n+1}}$.  Further, it intersects $G_{n-1}$, and hence  $\langle {\bf u}_{n-1} \rangle $, trivially.  Lastly, $\langle {\bf u}_{n-1} \rangle $ is isomorphic to $F_2$ since it consists of stable letters introduced in the construction of $G_{n-1}$.  It follows that 
$\langle \stable n1, \ldots, \stable n{2^n}, {\bf u}_{n-1}\rangle$
of $H_n$ is isomorphic to $F_{2^{n+1}} \times F_2$, where the $F_2$ factor has basis ${\bf u}_{n-1}$.  

{\it Suspended wings:}
We first analyze the subgroups $\langle \mathcal L_{n+1}(i) \rangle$.
Note that for $k \leq n$, the subgroups 
$\langle \stable ki \rangle$, where $1 \leq i \leq 2^{k}$, and 
$\langle {\bf u}_k \rangle$ of $G_{n}$ are isomorphic to $F_2$ by 
Remark~\ref{rem:stablefree}.  Then the diagonal subgroups 
$\langle {\bf u}_{k-1}^{-1}\stable ki\rangle$ are isomorphic to $F_2$ as well.  Since every
vector $\mathcal L_{n+1}(i)$ is either $\stable ki$ or ${\bf u}_{k-1}^{-1}\stable ki$ for some $k \leq n-1$, the subgroup $\langle \mathcal L_{n+1}(i)\rangle$ of $G_{n}$ is isomorphic to $F_2$ for $1 \leq i \leq 2^{n+1}$.

Note that, as shown above, the edge group in the construction of $G_{n}$ is isomorphic to
$\langle \stable {n}j \rangle_{j=1}^{2^{n}} \times \langle \mathbf u_{n-1} \rangle$.  In particular, $\langle \mathcal L_{n+1}(i) \rangle$ does not intersect the edge group, and 
Observation~\ref{obs:freeproduct} implies that $\langle {\bf u}_{n},  
\mathcal L_{n+1}(i)\rangle$ is isomorphic to $F_4$ for all $1 \leq i \leq 2^{n+1}$.

{\it Inductive step for (2):}
Let $1\leq i \leq 2^{n}$.
By construction, the stable letters ${\bf u}_{n}$ for $G_{n}$ act on $\stable {n}i$ via $\varphi$. The vector $\mathcal L_{n+1}(i)$ is the same as $\mathcal L_{n}(i)$, and 
the fact that $\mathcal L_{n}(i)$ commutes with $\stable {n}i$ follows from the construction of $H_{n}$.  This shows that $\langle {\bf u}_{n}, \mathcal L_{n+1}(i)\rangle$  acts on 
$\langle \stable {n}i \rangle$ via $\theta$. 
Since $G_n= M(M(G_{n-1},  \stable n1 \ldots \stable n{2^n} ), {\bf u}_n)$, 
 Lemma~\ref{lem:trivialintersection}  (with ${\bf b}$ trivial and ${\bf c}= \mathcal L_{n+1}(i)$)
 shows that 
$\langle {\bf u}_{n}, \mathcal L_{n+1}(i)\rangle \cap \langle \stable {n}i \rangle$ is trivial.
 Thus the subgroup $\langle \stable {n}i, \mathbf u_{n}, \mathcal{L}_{n+1}(i) \rangle$ of $G_{n}$ is isomorphic to $F_2 \rtimes_{\theta} F_4$.

{\it Inductive step for (3):} Let $1\leq j \leq 2^{n}$. Then 
${\bf u}_{n}$ acts on $\stable {n}j$ and
${\bf u}_{n-1}$ (and hence on ${\bf u}_{n-1}^{-1}\stable {n} j$) via $\varphi$.
Recall that 
$$
\mathcal{L}_{n+1}(j+2^{n}) =
\begin{cases}
 \stable{(n-1)}j  & \text{for} \;\; 1 \leq j \leq 2^{n-1}  \\
{\bf u}_{n-2}^{-1} \stable{(n-1)}{(j-2^{n-1})}  & \text{for}\;\; 2^{n-1} < j \leq 2^{n} \\
\end{cases}
$$
and observe (from the definitions of $G_{n-1}$ and $H_{n}$) that both ${\bf u}_{n-1}$ and $\stable {n}j$ act on 
$\mathcal L_{n+1}(j+2^{n})$ via $\varphi$. 
It follows that $\mathcal{L}_{n+1}(j+2^{n})$ commutes with $\mathbf u_{n-1}^{-1} \stable {(n)}j$.
Thus 
$\langle \mathbf u_{n}, \mathcal{L}_{n+1}(j+2^{n}) \rangle$ acts on 
$\langle{\bf u}_{n-1}^{-1}\stable {n} j\rangle$ via $\theta$. The fact that these two 
groups have trivial intersection follows from Lemma~\ref{lem:trivialintersection} (taking ${\bf b}= {\bf u}_{n-1}^{-1}$ and 
${\bf c}=  \mathcal{L}_{n+1}(j+2^{n})$).  Thus 
$\langle \mathbf u_{n-1}^{-1} \stable {n}j, \mathbf u_{n}, \mathcal{L}_{n+1}(j+2^{n}) \rangle$ is isomorphic to 
$F_2 \rtimes_{\theta} F_4$.
\end{proof}

\subsection{The cell structure in $K_{G_n}$ and $K_{H_n}$}
\label{sec:cells}
In this section we describe $3$-dimensional $K(\pi, 1)$'s for 
$G_n$ and $H_n$. 
Start with the standard cell structure (one $0$-cell and two $1$-cells) on a bouquet of two circles.  The product of three copies of this with the product cell structure is defined to be $K_{H_0}$. 

Since $G_n$ and $H_n$ are defined recursively as graphs of groups with 
$2$-dimensional edge groups and $3$-dimensional vertex groups, the complexes $K_{G_n}$ and $K_{H_n}$ are constructed inductively as total spaces of graphs of $3$-dimensional vertex spaces and $2$-dimensional edge spaces.

Here is the model situation.  
Let $X^{(3)}$ be a $3$-dimensional cell complex and let $Y_1^{(2)}, \ldots Y_k^{(2)}\subset X^{(3)}$ be $\pi_1$-injective $2$-dimensional subcomplexes. 
For each $i$, let $\Phi_i\!:Y_i^{(2)} \to Y_i^{(2)}$ be a cellular map inducing an automorphism in $\pi_1$.  We define a new $3$-complex $Z^{(3)}$ by 
\[
Z^{(3)} \; = \;\left. X^{(3)} \, \sqcup \, \left( \sqcup_i Y_i^{(2)}\times [0,1] \right)  \,
\right/ \sim 
\]
where $\sim$ identifies $Y_i^{(2)} \times 0$ with $Y_i^{(2)}$ (via the inclusion map) and $Y_i^{(2)} \times 1$ with $\Phi(Y_i^{(2)})$ for $1 \leq i \leq k$.
Each $Y_i^{(2)} \times [0,1]$ is given the product cell structure with a subdivision at the $Y_i^{(2)} \times 1$ end, obtained by pulling back the standard cell structure on $Y_i^{(2)}$ via $\Phi_i$.  

To get $K_{G_n}$ from $K_{H_n}$ we apply this model situation, taking $X^{(3)}$ to be $K_{H_n}$ and the $Y_i^{(2)}$ to be the subcomplexes corresponding to the edge groups listed in the $G_n$-row of the third column of Table~\ref{groups-table}.  

To get $K_{H_{n}}$ from $K_{G_{n-1}}$ we first apply this model situation taking $X^{(3)}$ to be $K_{G_{n-1}}$ and the $Y_i^{(2)}$ to be the subcomplexes corresponding to the edge groups listed in the $H_{n}$-row of the third column of Table~\ref{groups-table}.  We subdivide the resulting $3$-complex by introducing edges labeled ${\bf u}_{n-1}^{-1}\stable{n}{j}$ and $2$-cells corresponding to the commutation relations $[{\bf u}_{n-1}^{-1} \stable {n}j, \mathcal {L}_{n+1}(j+2^{n})]=1$. 
 (Refer to the edge groups in the $H_{n+1}$-row of Table~\ref{groups-table} for these commutation relations.) 
 Note that the $\stable{n}{j}$ run over the stable letters of $H_{n}$. Figure~\ref{fig:subdivision} demonstrates such subdivisions in the case of $K_{H_1}$.

\begin{figure}[ht]
\begin{center}
\psfrag{a}{{\small $a_{121}$}}
\psfrag{u}{{\small $u_{0}$}}
\psfrag{d}{{\small $u_{0}^{-1}a_{1ji}$}}
\psfrag{0}{{\small $a_{021}$}}
\psfrag{1}{{\small $a_{022}$}}
\psfrag{f}{{\small $a_{021}$}}
\psfrag{p}{{\small $\varphi(a_{021})$}}
\psfrag{2}{{\small $\varphi^2(a_{021})$}}
\includegraphics[width=115mm]{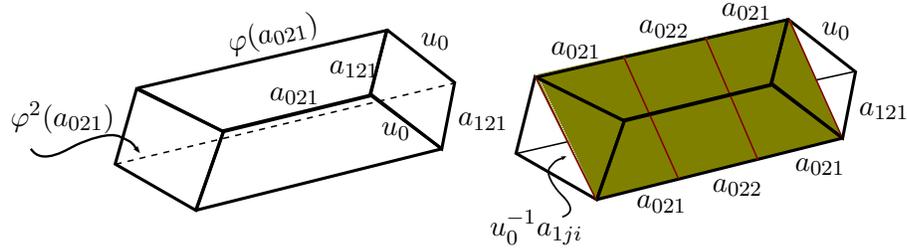}
\end{center}
\caption{Typical new 1- and 2-cells in $\widetilde{K_{H_1}}$.
\label{fig:subdivision}}
\end{figure}

\section{Lower bounds} \label{sec:lower}

\subsection{Some spheres}\label{somespheres}

As preparation for the general case, we explicitly describe some spheres in the first few groups. The computations of area and volume will be only roughly sketched, to help give a general idea. Formal verification of these computations will come later.

\medskip

\noindent{\bf Spheres in $G_0$.}
Fix $N\in\N$ and consider the slab in $\widetilde{K_{H_0}}$ (the universal cover of $K_{H_0}$) of the form $\varphi^{N}(a_{011})\times\varphi^{N}(a_{021})\times y_{1}$ (the equatorial slab in Figure~\ref{fig:G0basic} below). The volume of this slab is $|\varphi^N(\first)|^2$ (cf.~Lemma~\ref{lem:slab}).

Each of the top and bottom faces of this slab consists of 2-cells corresponding to relations in the edge group $\langle {\bf{a}}_{01}\rangle\times\langle{\bf{a}}_{02}\rangle$ in the definition of $G_0$. 
Allowing $u_{01}^{-N}$ to act on both faces produces the 3-ball shown in Figure~\ref{fig:G0basic}.
\begin{figure}[ht]
\psfrag{y}{{\small $y_1$}}
\psfrag{q}{{\small $\varphi^{N}(a_{021})$}}
\psfrag{b}{{\small $a_{021}$}}
\psfrag{a}{{\small $a_{011}$}}
\psfrag{u}{{\small $u_{01}^N$}}
\psfrag{p}{{\small $\varphi^{N}(a_{011})$}}
\begin{center}
\includegraphics[width=2.5in]{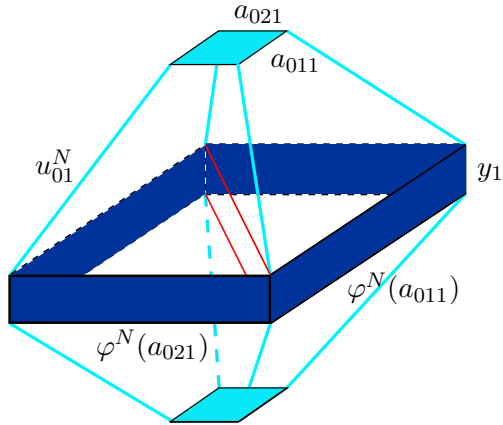}
\end{center}
\caption{The sphere $S_1(N)$ in $\widetilde{K_{G_0}}$.}
\label{fig:G0basic}
\end{figure}
Let us call the boundary sphere of this ball $S_1(N)$ in correspondence to Figure~\ref{fig:iterate}. Each of the eight trapezoidal faces of $S_1(N)$ is a van Kampen diagram for an equality of the form $\varphi^N(a_{0j1})=u_{01}^Na_{0j1}u_{01}^{-N}$. In particular, each vertical side of the trapezoid has length $N$, the top has length one, and the bottom has length $|\varphi^N(\first)|$, which is exponential in $N$. Thus the transition from the slab to $S_1(N)$ corresponds to two Type II moves, one on each face of the slab.

We can think of each trapezoidal face on $S_1(N)$ as a stack of horizontal strips. Because of the exponential growth of $\varphi$, the lengths of these strips approximate a geometric series (Lemma~\ref{lem:thetadehn}). Thus the area of a trapezoid is approximately equal to the length of its base edge, which is the same as the length of one piece of the equatorial band about the sphere. We deduce that ${\text{Area}}(S_1(N))\simeq |\varphi^N(\first)|$ (Lemma~\ref{lem:Gninduction}). The volume of the 3-ball is bounded below by the volume of the slab, which is $|\varphi^N(\first)|^2$, from which it follows that $\dehntwo{G_0}{x}\succeq x^2$ (cf.~Remarks~\ref{rem:embedded} and~\ref{rem:sparse}).

\medskip
\noindent{\bf Spheres in $H_1$.} 
Looking at $S_1(N)$ from a different point of view, we see that it consists of four belted-trapezoid pairs, each of which has the geometry marking the beginning of a Type I move (see Figure~\ref{fig:metricsub}). Each such trapezoid pair is contained in a group of the form $\langle{\bf a}_{0j}\rangle\rtimes_{\theta}\langle {\bf u}_0,{\bf y}\rangle$, an edge group in $H_1$. We again think of each trapezoid as comprised of a stack of horizontal strips with boundary labeled $u_{01}\, \varphi^{i-1}(a_{0j1})\, u_{01}^{-1}\, \varphi^i(a_{0j1}^{-1})$, where $1\leq i \leq N$.  The result of allowing $a_{1j1}^{-i}$ to act on each such strip and $a_{1j1}^{-N}$ to act on the four belts is the 
3-ball shown in Figure~\ref{fig:H1}. We call the boundary of this ball $S_2(N)$.

\begin{figure}[ht]
\psfrag{a}{{\small $y_1$}}
\psfrag{b}{{\small $(a_{111}^{-1}u_{01})^N$}}
\psfrag{c}{{\small $a_{111}^N$}}
\psfrag{d}{{\small $a_{021}$}}
\psfrag{e}{{\small $a_{121}^N$}}
\psfrag{f}{{\small $(a_{121}^{-1}u_{01})^N$}}
\psfrag{g}{{\small $a_{011}$}}
\begin{center}
\includegraphics[width=3in]{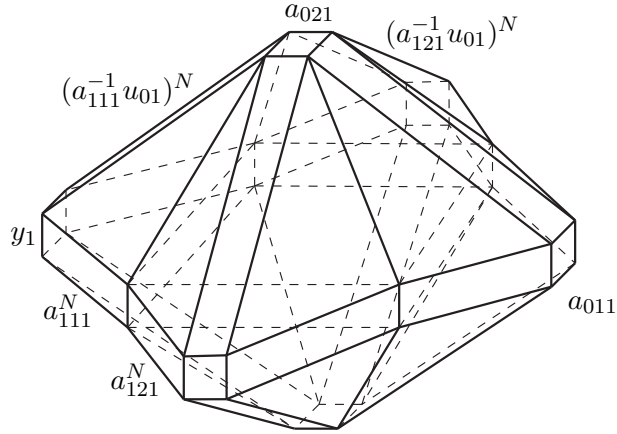}
\end{center}
\caption{The sphere $S_2(N)$ in $\widetilde{K_{H_1}}$.}
\label{fig:H1}
\end{figure}
The long sides of the horizontal strips in 
Figure~\ref{fig:H1} 
represent words of the form $a_{1j1}^N$, giving each such strip an area equal to $N$. The angled strips connected to the poles have long sides representing words of the form $(a_{1j1}^{-1}u_{01})^N$, so that these strips also have area equal to $N$ (recall from Section~\ref{sec:cells} that edges labeled $a_{1j1}^{-1}u_{01}$ are introduced into the cell structure on $K_{H_1}$). Each triangular region is a van Kampen diagram for an equation of the form $a_{1j1}^{-N}u_{01}^N=(a_{1j1}^{-1}u_{01})^N$ (recall that $a_{1j1}$ commutes with $u_{01}$). The area of such a triangle is essentially the number of commutation relations required to achieve this equality, which is of order $N^2$ (Lemma~\ref{lem:deltadehn}). Thus the area of $S_2(N)$ is of order $N^2$. On the other hand, the volume of the 3-ball is again bounded below by the volume of the slab, which is $|\varphi^N(\first)|^2\simeq (e^N)^2\simeq e^N$. It follows that $\dehntwo{H_1}{x}\succeq e^{\sqrt{x}}$.

\medskip

\noindent{\bf Spheres in $G_1$.}
To form $S_3(N)$ in $\widetilde{K_{G_1}}$, we alter the previous constructions slightly. We still begin with a slab of the form $\varphi^N(a_{011})\times\varphi^N(a_{021})\times y_1$, but now we set $N=|\varphi^M(\xi)|$ for some $M$, so that the volume of the slab, as a function of $M$, is $|\varphi^{|\varphi^M(\xi)|}|\simeq e^{e^M}$. We next allow the word $\varphi^M(u_{01}^{-1})$ (instead of $u_{01}^{-N}$) to act on the top and bottom faces of the slab, producing a sphere analogous to the one in Figure~\ref{fig:G0basic}. Performing Type I moves as before, but with $\varphi^M(a_{1j1})$ acting on the belts instead of $a_{1j1}^N$, produces a sphere analogous to the one in Figure~\ref{fig:H1}, but with labels changed as indicated in Figure~\ref{fig:G1}.
\begin{figure}[ht]
\psfrag{a}{{\small $\varphi^M(a_{111})$}}
\psfrag{A}{{\small $\varphi^M(a_{121})$}}
\psfrag{u}{{\small $\varphi^M(u_{01})$}}
\psfrag{d}{{}}
\psfrag{D}{{}}
\begin{center}
\includegraphics[width=3in]{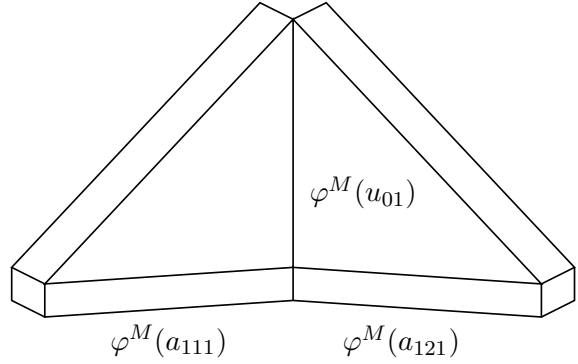}
\end{center}
\caption{A portion of a variation of the sphere in $\widetilde{K_{H_1}}$ shown in Figure~\ref{fig:H1}. The triangle pair admits an action by $u_{11}^{-M}$. Note that $\varphi^M(u_{01})$ and $\varphi^M(a_{1j1}^{-1})$ have the same word length. The labels on the diagonal edges in the figure above are formed by alternating the letters of these words.
\label{fig:G1}}
\end{figure}
The triangles in this sphere can be grouped in pairs to form eight quadrilaterals, each with the geometry marking the beginning of a Type II move. Each of these quadrilaterals is 
a van Kampen diagram labeled by elements of the group $\langle {\bf a}_{11},{\bf a}_{12}\rangle\times\langle{\bf u}_0\rangle$, the edge group for $G_1$. We now allow $u_{11}^{-M}$ to act on each of these quadrilaterals. As in $G_0$, this will produce trapezoids representing equalities of the form $u_{11}^{M}gu_{11}^{-M}=\varphi^M(g)$, where $g$ is of the form either $a_{1j1}$ or $a_{1j1}^{-1}u_{01}$. These relations have area of order $|\varphi^M(\first)|\simeq e^M$, which gives the area of the sphere. The volume of the slab is $\simeq e^{e^M}$, so that $\dehntwo{G_1}{x}\succeq e^{x}$.

\subsection{Iteration notation}\label{sec:iteration}

Before moving on to the general case, we pause briefly to develop the iteration notation necessary to avoid unwieldy towers of exponents.

Recall that the map $\varphi\!:\langle \first,\second \rangle\to\langle \first, \second \rangle$ has the form $\varphi(\first)=\first \second \first$ and $\varphi(\second)=\first$. We inductively define functions $w_n\!:\N\to\N$, setting $w_0(r)=r$ for all $r\in\N$, and
$$
w_n(r)=|\varphi^{w_{n-1}(r)}(\first)|.
$$
The growth of $|\varphi^N(\first)|$ is determined by the matrix $\begin{bmatrix} 2&1\\ 1&0\end{bmatrix}$, which has positive eigenvalue $1+\sqrt{2}$. More precisely, we have that
$$
|\varphi^N(\first)|=\left\|\begin{bmatrix}2&1\\1&0\end{bmatrix}^N\!\begin{bmatrix}1\\0\end{bmatrix}\right\|_1\approx\frac{1}{2}(1+\sqrt{2})^{N+1}.
$$
This implies that
$$
\lim_{k\to\infty}\frac{|\varphi^{k+1}(\first)|}{|\varphi^{k}(\first)|}=1+\sqrt{2}.
$$
It follows easily that
$$
\sum_{k=0}^{N}|\varphi^k(\first)|\simeq |\varphi^N(\first)|,
$$
as this sum is approximately geometric. From this discussion we deduce the following.
\begin{lemma}[Growth of $w_n$]\label{lem:growth}
For each $n$ we have
$$
\sum_{i=0}^{w_{n-1}(r)}|\varphi^i(\first)|\simeq w_n(r)\simeq\exp^n(r).
$$
\end{lemma}

\subsection{The van Kampen diagrams for moves}\label{sec:moves}

We now describe the van Kampen diagrams that will make up the surface of our spheres, and we explain how they correspond to the schematic diagrams and the subdivision-combination moves.

Given two words $W_1=b_1b_2\cdots b_m$ and $W_2=c_1c_2\cdots c_m$, we let $\delta(W_1,W_2)$ denote the word $b_1c_1b_2c_2\cdots b_mc_m$. With this notation, we define $\ddii{ij}{n}{w_k(r)}$, for $k\geq 1$, $n\geq 0$, and $1\leq i,j\leq 2^n$, to be the van Kampen diagram in $\widetilde{K_{H_n}}$ shown in Figure~\ref{fig:deltadehn}.
\begin{figure}[h]
\begin{center}
\psfrag{b}{\footnotesize{$\delta\Big(\varphi^{w_{k-1}(r)}(a_{ni1}^{-1}),\varphi^{w_{k-1}(r)}(u_{(n-1)1})\Big)$}}
\psfrag{d}{\footnotesize{$\delta\Big(\varphi^{w_{k-1}(r)}(a_{nj1}^{-1}),\varphi^{w_{k-1}(r)}(u_{(n-1)1})\Big)$}}
\psfrag{i}{\footnotesize{$\varphi^{w_{k-1}(r)}(a_{ni1})$}}
\psfrag{j}{\footnotesize{$\varphi^{w_{k-1}(r)}(a_{nj1})$}}
\psfrag{u}{\footnotesize{$\varphi^{w_{k-1}(r)}\big(u_{(n-1)1}\big)$}}
\includegraphics[width=3.5in]{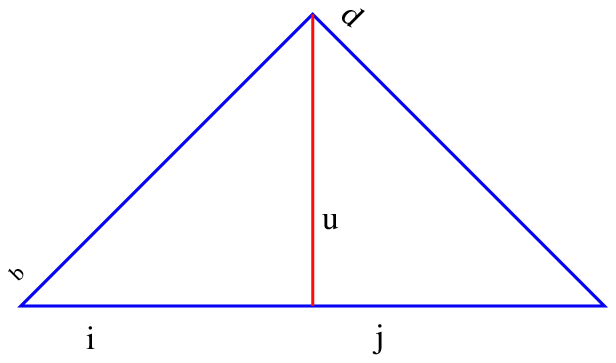}
\end{center}
\caption{The van Kampen diagram $\ddii{ij}{n}{w_k(r)}$. \label{fig:deltadehn}}
\end{figure}

Such a van Kampen diagram will be found in the initial stage of a Type II move. The following lemma establishes the existence and geometry of such van Kampen diagrams.
\begin{lemma}\label{lem:deltadehn}
The van Kampen diagram $\ddii{ij}{n}{w_k(r)}$ exists. Moreover, $\mathrm{Area}(\ddii{ij}{n}{w_k(r)})\simeq [w_k(r)]^2$.
\end{lemma}

\begin{proof}
Because $\varphi$ is palindromic, we may write
$$
\varphi^{w_{k-1}(r)}(a_{ni1})=b_1\cdots b_{p} b_{p+1}b_{p}\cdots b_1,
$$
where for $j\neq p+1$ we have $b_j\in\{a_{ni1}^{\pm},a_{ni2}^{\pm}\}$, and $b_{p+1}\in\{a_{ni1}^{\pm},a_{ni2}^{\pm},e\}$.
We similarly write
$$
\varphi^{w_{k-1}(r)}(u_{(n-1)1})=c_1\cdots c_{p}c_{p+1}c_{p}\cdots c_1.
$$
Recall from Table~\ref{groups-table} that $a_{nij}$ commutes with $u_{(n-1)j}$ for 
$1 \leq i \leq 2^n$ and $j=1,2$.  
Then
\begin{align*}
\varphi^{w_{k-1}(r)}(a_{ni1}^{-1})\varphi^{w_{k-1}(r)}&(u_{(n-1)1}) \\
&= (b_1\cdots b_{p} b_{p+1}b_{p}\cdots b_1)^{-1}(c_1\cdots c_{p}c_{p+1}c_{p}\cdots c_1) \\
&= 
(b_1^{-1}c_1)\cdots (b_p^{-1}c_p)(b_{p+1}^{-1}c_{p+1})(b_p^{-1} c_p)\cdots (b_1^{-1} c_1) \\
&=\delta(\varphi^{w_{k-1}(r)}(a_{ni1}^{-1}),\varphi^{w_{k-1}(r)}(u_{(n-1)1})).
\end{align*}

The area claim follows by counting the number of commutation relations involved in this equality, remembering that each such relation provides two 2-cells due to the diagonal subdivision in the cell structure. We obtain
$$
\mathrm{Area}(\ddii{ij}{n}{w_k(r)})=2\left(\sum_{j=1}^{2p} j+(2p+1)+\sum_{j=1}^{2p}j\right)
$$
$$
=2(2p+1)^2=2|\varphi^{w_{k-1}(r)}(g_1)|^2\simeq [w_k(r)]^2.
$$
\end{proof}

\begin{remark}\label{rem:actions}
The reason $\varphi$ is chosen to be palindromic is to
ensure the existence of such van Kampen diagrams. 
To understand the mechanics behind $\ddii{ij}{n}{N}$, 
consider the subgroup $\langle g_1^{-1}h_1,g_2^{-1}h_2\rangle\cong F_2$ of the group $\langle g_1,g_2\rangle\times\langle h_1,h_2\rangle\cong F_2\times F_2$.  
The action of $\varphi \times \varphi$ on $F_2 \times F_2$ restricts to an automorphism of this subgroup exactly when $\varphi$ is palindromic, and in this case the induced action is the same as the action of $\varphi$ coming from the isomorphism 
$\langle g_1^{-1}h_1,g_2^{-1}h_2\rangle \simeq F_2$.  Note that the diagonal sides of $\ddii{ij}{n}{w_k(r)}$ are labeled by the generators of such subgroups, and are in fact of the form 
$\varphi^{{w_{k-1}(r)}}(g_1^{-1}h_1)$.  This fact is crucial for performing future generations of the iterative procedure. 
\end{remark}

When $k=0$ we define $\ddii{ij}{n}{w_0(r)}$ to be the van Kampen diagram obtained from the one in Figure~\ref{fig:deltadehn} by replacing all labels of the form $\varphi^{w_{k-1}(r)}(g)$ by $g^r$. Also the notation $\ddii{12}{0}{w_k(r)}$ will make sense with the definition above if we define ${\bf{u}}_{-1}={\bf{a}}_{01}{\bf{a}}_{02}^{-1}$. For then we have $\delta(\varphi^{w_{k-1}(r)}(a_{011}^{-1}),\varphi^{w_{k-1}(r)}(a_{011}a_{021}^{-1}))=\varphi^{w_{k-1}(r)}(a_{021}^{-1})=[\varphi^{w_{k-1}(r)}(a_{021})]^{-1}$. Each of these diagrams has geometry identical to that of the others, up to $\simeq$-equivalence. See Figure~\ref{fig:deltadehn1}.
\begin{figure}[ht]
\begin{center}
\psfrag{3}{\tiny{$\varphi^{w_{k-1}(r)}(a_{011}^{-1}a_{021})$}}
\psfrag{1}{\tiny{$\varphi^{w_{k-1}(r)}(a_{011})$}}
\psfrag{2}{\tiny{$\varphi^{w_{k-1}(r)}(a_{021})$}}
\psfrag{4}{\tiny{$a_{021}$}}
\psfrag{5}{\tiny{$a_{022}$}}
\psfrag{6}{\tiny{$a_{011}$}}
\psfrag{7}{\tiny{$a_{012}$}}
\includegraphics[width=120mm]{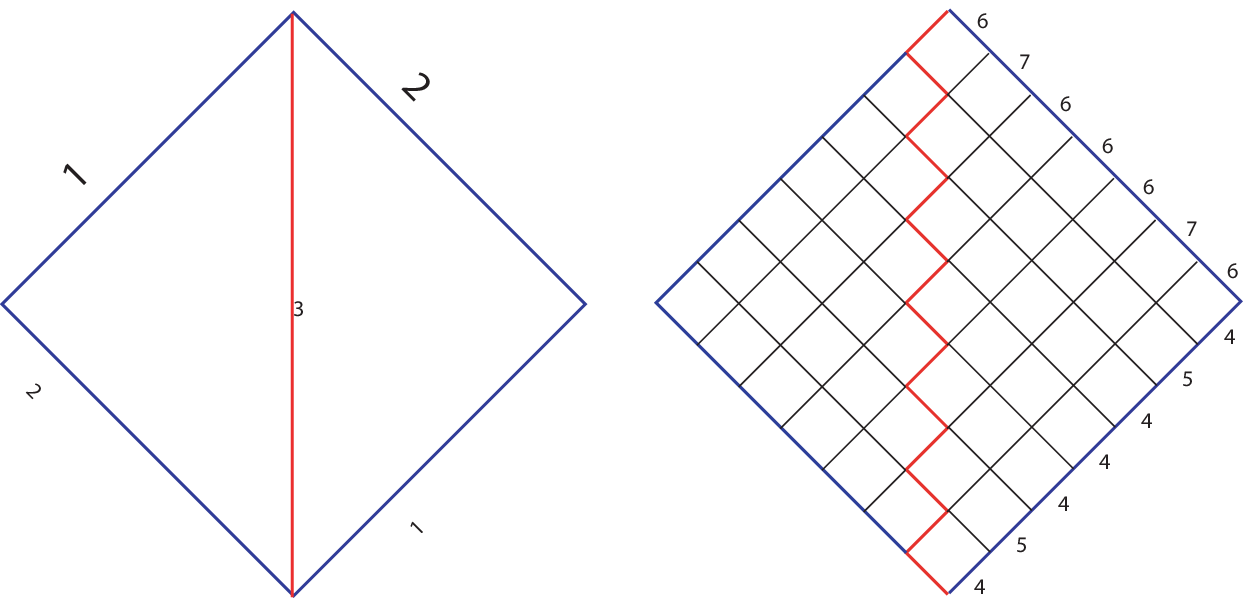}
\end{center}
\caption{The van Kampen diagrams $\ddii{12}{0}{w_k(r)}$ and $\ddii{12}{0}{w_0(2)}$. \label{fig:deltadehn1}}
\end{figure}

For the other van Kampen diagram with which we will be concerned, we first recall that $\mathcal{L}_n(i)$ denotes the $i$th element of the list $\mathcal{L}_n$. This element is a 2-vector $\langle \first,\second \rangle$. It will be useful in what follows to let $\mathcal{L}_n^{(1)}(i)$ denote the first element $\first$ of this vector. With this notation we define $\ddi{i}{n}{w_k(r)}$, for $k \geq 1$, $n\geq 0$, and $1\leq i\leq 2^n$, to be the van Kampen diagram in $\widetilde{K_{G_n}}$ shown in Figure~\ref{fig:thetadehn}. (Recall that each vector in the list $\mathcal L_{n+2}$ consists of elements in $G_n$.)
\begin{figure}[h]
\begin{center}
\psfrag{a}{\footnotesize{$\mathcal{L}_{n+2}^{(1)}(i+2^{n+1})$}}
\psfrag{y}{\footnotesize{$\mathcal{L}_{n+2}^{(1)}(i)$}}
\psfrag{u}{\footnotesize{$\varphi^{w_{k-1}(r)}(u_{n1})$}}
\psfrag{d}{\footnotesize{$\varphi^{w_k(r)}\big(\mathcal{L}_{n+2}^{(1)}(i+2^{n+1})\big)$}}
\psfrag{A}{\footnotesize{$a_{121}$}}
\psfrag{U}{\footnotesize{$u_{11}$}}
\psfrag{V}{\footnotesize{$u_{12}$}}
\psfrag{Y}{\footnotesize{$y_{1}$}}
\includegraphics[width=120mm]{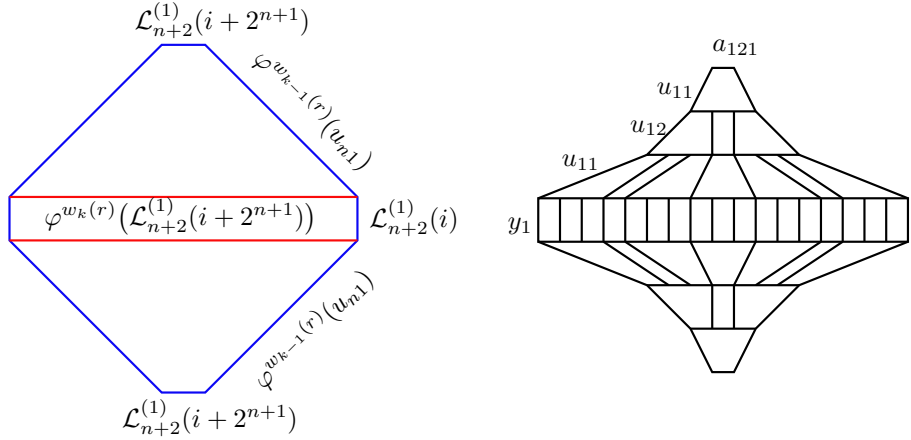}
\end{center}
\caption{The van Kampen diagrams $\ddi{i}{n}{w_k(r)}$ and $\ddi{2}{1}{w_0(1)}$. \label{fig:thetadehn}}
\end{figure}
When $k=0$ we define $\ddi{i}{n}{w_0(r)}$ to be the van Kampen diagram obtained from that in Figure~\ref{fig:thetadehn} by replacing all labels of the form $\varphi^{w_{k-1}(r)}(g)$ by $g^r$.

We establish the existence and geometry of such diagrams in the following lemma.
\begin{lemma}\label{lem:thetadehn}
The van Kampen diagram $\ddi{i}{n}{w_k(r)}$ exists. Moreover, $\mathrm{Area}(\ddi{i}{n}{w_k(r)})\simeq w_{k+1}(r)$.
\end{lemma}

For the intuition behind this result, see $\ddi{2}{1}{w_0(1)}$ in Figure~\ref{fig:thetadehn}. The crucial ingredient for the area claim is the exponential growth of $\varphi$.

\begin{proof}
To establish the existence of the van Kampen diagram, note that 
each of $u_{n1}$ and $u_{n2}$ acts on 
$\mathcal{L}_{n+2}^{(1)}(i+2^{n+1})$ via $\varphi$.
Thus as elements of $G_n$, the word 
$\varphi^{w_{k-1}(r)}(u_{n1}) \mathcal{L}_{n+2}^{(1)}(i+2^{n+1}) \varphi^{w_{k-1}(r)}(u_{n1}^{-1})$  and the word labeling the long sides of the central strip in $\ddi{i}{n}{w_k(r)}$ are equal.  This establishes the top and bottom trapezoids
 of $\ddi{i}{n}{w_k(r)}$.

For the middle strip, we simply need to observe that $\mathcal{L}_{n+2}^{(1)}(i)$ commutes with $\mathcal{L}_{n+2}(i+2^{n+1})$.  To see this, note that each $\mathcal{L}_{n+2}(i+2^{n+1})$ term is also the first factor in an edge group for $H_{n+1}$, while each $\mathcal{L}_{n+2}^{(1)}(i)$ term is the third factor in the corresponding group. These commute, by definition of the action of $\theta$.

We now compute the area of $\ddi{i}{n}{w_k(r)}$. 
Note that the area obtained when $\varphi$ acts on a positive word $W$ is exactly $|W|$. 
Since $\varphi$ is applied 
$|\varphi^{w_{k-1}(r)}(\first)| \simeq w_k(r)$ times in each trapezoid, its
area is
$$
1+|\varphi(\first)|+|\varphi^2(\first)|+\cdots+|\varphi^{w_{k}(r)-1}(\first)|,
$$
so that $\mathrm{Area}(\ddi{i}{n}{w_k(r)})\simeq |\varphi^{w_k(r)}(\first)|=w_{k+1}(r)$, as required.
\end{proof}

The geometry of the van Kampen diagrams can be summarized as shown in 
Figure~\ref{fig:summary}, where the labels on the edges indicate lengths.
\begin{figure}[ht]
\begin{center}
\psfrag{0}{\footnotesize{$w_{k}(r)$}}
\psfrag{1}{\footnotesize{$w_{k+1}(r)$}}
\psfrag{A}{\footnotesize{Area$\,\simeq [w_{k}(r)]^2$}}
\psfrag{B}{\footnotesize{Area$\,\simeq w_{k+1}(r)$}}
\includegraphics[width=3.5in]{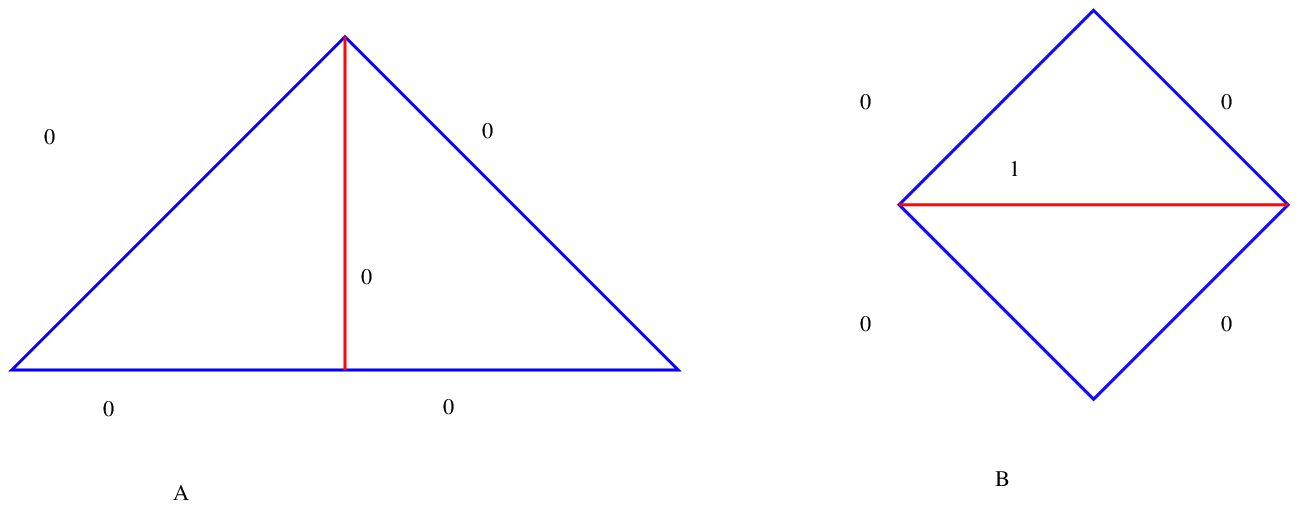}
\end{center}
\caption{The geometry of $\ddii{ij}{n}{w_k(r)}$ and $\ddi{i}{n}{w_k(r)}$. \label{fig:summary}}
\end{figure}

\subsection{Geometry of the moves}

\subsubsection{Type II}

As indicated above, the geometry of the van Kampen diagram $\ddii{ij}{n}{w_k(r)}$ is that of the schematic quadrilateral with which a Type II move begins, in that each side has equivalent length, and the area is equivalent to the square of this length. We now explain how such a move is performed in the group $G_n$. Note that $\ddii{ij}{n}{w_k(r)}$ is contained in the edge group for $G_n$. As such it admits an action by the stable letters ${\bf u}_n$.
\begin{lemma}\label{lem:typeii}
Assume $k\geq 1$, and suppose $W$ is a positive word in the $u_{nj}$ of length $w_{k-1}(r)$, and let $B$ denote the abstract 3-ball obtained by allowing $W^{-1}$ to act on $\ddii{ij}{n}{w_k(r)}$. Let $S$ denote the 2-sphere boundary of $B$. Then
$$
\mathrm{Vol}(B)\succeq w_{k}^2(r)
$$
and
$$
\mathrm{Area}(S-\ddii{ij}{n}{w_k(r)})\simeq w_{k}(r).
$$
\end{lemma}

\begin{figure}[ht]
\psfrag{a}{$\rightsquigarrow$}
\psfrag{s}{\footnotesize{$\varphi^{w_{k-2}(r)}(u_{n1})$}}
\psfrag{b}{\footnotesize{$\varphi^{w_{k-1}(r)}(a_{ni1})$}}
\begin{center}
\includegraphics[width=3.5in]{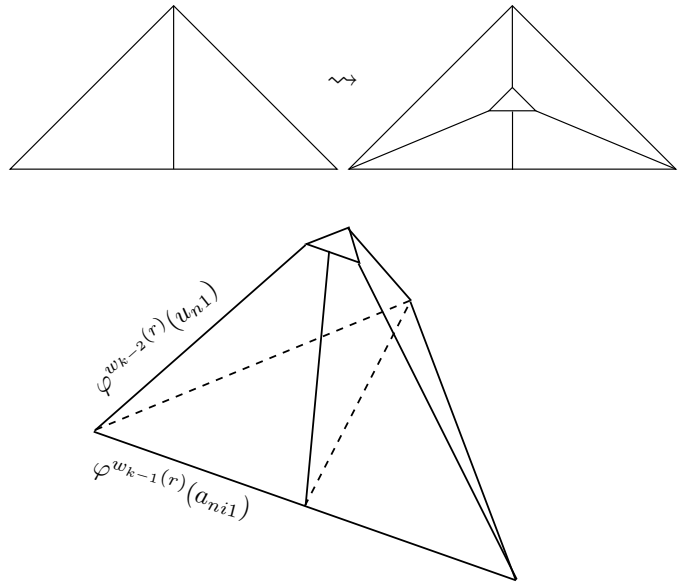}
\end{center}
\caption{A view from above of the action of $W$ on $\ddii{ij}{n}{w_k(r)}$, as well as the resulting 
$3$-ball in the case that $W=\varphi^{w_{k-2}(r)}(u_{n1})$. Note that each vertical quadrilateral on 
the boundary of the $3$-ball is identical to a portion of $\ddi{i}{n}{w_{k-1}(r)}$.
\label{fig:cone}}
\end{figure}

\begin{proof}
Note that $w_{k}^2(r)$ is exactly the area of the van Kampen diagram $\ddii{ij}{n}{w_k(r)}$. In the bottom layer of the 3-ball there is at least one 3-cell for every nine 2-cells in $\ddii{ij}{n}{w_k(r)}$. The volume claim follows.

The boundary of $\ddii{ij}{n}{w_k(r)}$ consists of four words of the form
 $\varphi^{w_{k-1}(r)}(g)$ for some $g\in\{a_{ni1},a_{nj1},a_{ni1}^{-1}u_{(n-1)1},a_{nj1}^{-1}u_{(n-1)1}\}$. The action of $W^{-1}$ on each of these words has the geometry of one trapezoid of a $\ddi{i}{n}{w_{k-1}(r)}$. As such the area claim follows from the argument in the proof of  Lemma~\ref{lem:thetadehn}.
\end{proof}

\begin{remark}
When $k\geq 2$ and $W=\varphi^{w_{k-2}(r)}(u_{n1})$, the action of $W^{-1}$ on each of the four boundary words of $\ddii{ij}{n}{w_k(r)}$ produces precisely a trapezoid of $\ddi{i}{n}{w_{k-1}(r)}$, as shown in Figure~\ref{fig:cone}. When $k=1$ and $W=u_{n1}^r$, it produces a trapezoid of $\ddi{i}{n}{w_0(r)}$.
\end{remark}

\subsubsection{Type I}

The geometry of $\ddi{i}{n}{w_k(r)}$ is that of the schematic quadrilateral with which a Type I move begins, in that each side of the quadrilateral has equivalent length, while the length of the diagonal and the total area are both exponential in the outer edge length. We now explain how a Type I move is performed on such a van Kampen diagram in $H_{n+1}$. Note that each $\ddi{i}{n}{w_k(r)}$ is contained in some single edge group for $H_{n+1}$ (depending on $i$). As such it admits an action by the stable letters ${\bf a}_{(n+1)j}$, for some $j$ depending on $i$.  In particular, given a word $V$ in the $a_{(n+1)ji}^{-1}$, we may perform the procedure analogous to the one 
performed on each belted trapezoid pair of $S_1(N)$ in the construction of $S_2(N)$ 
(cf.~Section~\ref{somespheres}).  In the following lemma, this is referred to as {\it allowing $V$ to act on $\ddi{i}{n}{w_k(r)}$}.

\begin{lemma}\label{lem:typei}
Suppose $W$ is a positive word in the $a_{nji}$ of length $w_k(r)$, and let $B$ denote the 3-ball in $H_{n+1}$ obtained by allowing $W^{-1}$ to act on $\ddi{i}{n}{w_k(r)}$. Let $S$ denote the 2-sphere boundary of $B$. Then
$$
\mathrm{Vol}(B)\succeq w_{k+1}(r),
$$
and
$$
\mathrm{Area}(S-\ddi{i}{n}{w_k(r)})\simeq [w_k(r)]^2.
$$
\end{lemma}

\begin{figure}[ht]
\psfrag{a}{$\rightsquigarrow$}
\psfrag{u}{\footnotesize{$\varphi^{w_{k-1}(r)}(u_{n1})$}}
\psfrag{d}{\footnotesize{$\varphi^{w_{k-1}(r)}(a_{(n+1)i1})$}}
\psfrag{s}{\footnotesize{$\delta\Big(\varphi^{w_{k-1}(r)}(a_{(n+1)i1}^{-1}),\varphi^{w_{k-1}(r)}(u_{n1})\Big)$}}
\begin{center}
\includegraphics[width=3.5in]{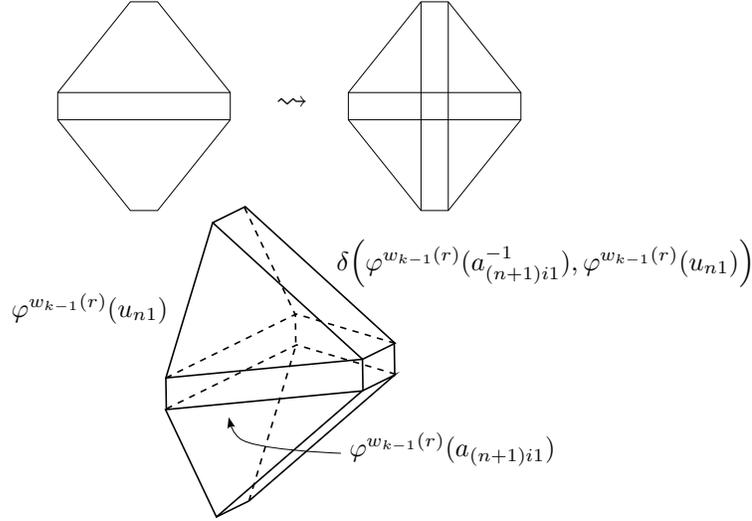}
\end{center}
\caption{A view from above of the action of $W$ on $\ddi{i}{n}{w_k(r)}$, as well as the resulting $3$-ball in the case that $W=\varphi^{w_{k-1}(r)}(a_{(n+1)i1})$. Note that each triangle on the boundary of the $3$-ball is identical to a portion of $\ddii{ij}{n+1}{w_{k}(r)}$.
\label{fig:wing}}
\end{figure}

\begin{proof}
The lower volume bound is trivial, as this is the length of the central strip, and there is at least one 3-cell for every three 2-cells in this strip.

Suppose first that $k\geq 1$. Then the boundary of $\ddi{i}{n}{w_k(r)}$ consists of four copies of the word $\varphi^{w_{k-1}(r)}(u_{n1})$ along with four words of length one (those from $\mathcal{L}_{n+2}$). The action of $\varphi^{w_{k-1}(r)}(a_{nj1}^{-1})$ on $\varphi^{w_{k-1}(r)}(u_{n1})$ has the geometry of one-half of a $\ddii{ij}{n}{w_k(r)}$ diagram. As such the area of each of the corresponding pieces of $S-\ddi{i}{n}{w_k(r)})$ is $\simeq [w_{k}(r)]^2$. Because each letter in $\varphi^{w_{k-1}(r)}(a_{nj1}^{-1})$ commutes with each of the letters in $\mathcal{L}_{n+2}$, the area resulting from the action on these words is $\simeq w_{k}(r)$. The area claim follows for this case.

For the case that $k=0$, the same argument applies with $\varphi^{w_{k-1}(r)}(g)$ replaced by $g^r$.
\end{proof}

\begin{remark}
When $W=\varphi^{w_{k-1}(r)}(a_{nj1})$, the action of $W^{-1}$ on $\varphi^{w_{k-1}(r)}(u_{n1})$ produces precisely one half of $\ddii{ij}{n}{w_k(r)}$.
\end{remark}

\subsection{Lower bound estimates}\label{base}

\noindent {\it{Base Cases.}} We now inductively compute lower bounds for $\dehntwo{G_n}x$ and $\dehntwo{H_n}x$. For the base cases, we return to the spheres described in Section~\ref{somespheres} and point out how the claims made there are verified by the computations just obtained. The only ingredient we lack is the following.
\begin{lemma}[Slab volume]\label{lem:slab}
The volume of the slab $\varphi^{w_k(r)}(a_{011})\times\varphi^{w_k(r)}(a_{021})\times y_{1}$ is $[w_{k+1}(r)]^2$.
\end{lemma}

\begin{proof}
As a cell-complex this slab is isomorphic to $\ddii{12}{0}{w_{k+1}(r)}\times [0,1]$ with the product cell structure. The result then follows from Lemma~\ref{lem:deltadehn}.
\end{proof}
For $G_0$ we take $k=0$, and so begin with the slab $\varphi^r(a_{011})\times\varphi^r(a_{021})\times y_1$ with volume $[w_1(r)]^2$. The boundary of this slab is a topological sphere containing two copies of $\ddii{12}{0}{w_1(r)}$. Let $B_{G_0}(r)$ denote the 3-ball obtained by allowing $u_{01}^{-r}$ to act on both of these. Let $S_{G_0}(r)$ denote the 2-sphere boundary of $B_{G_0}(r)$.

\begin{lemma}[Area and volume in $G_0$]
 $\mathrm{Area}(S_{G_0}(r))\simeq w_1(r)$ and $\mathrm{Vol}(B_{G_0}(r))\succeq [w_1(r)]^2$.
\end{lemma}

\begin{proof}
The volume claim follows immediately from the volume of the slab, computed above. For the area, we note that the band around the slab consists of four pieces, each with area $|\varphi^r(\first)|=w_1(r)$. The remaining area is shown to be $\simeq w_1(r)$ by applying Lemma~\ref{lem:typeii} with $k=0$.
\end{proof}

We now consider $H_1$. Note that $S_{G_0}(r)$ consists of four copies of $\ddi{i}{0}{w_0(r)}$ (two with $i=1$, two with $i=2$) plus two additional 2-cells. We define $B_{H_1}(r)$ to be the 3-ball obtained by allowing $a_{111}^{r}$ and $a_{121}^{r}$ to act on the corresponding van Kampen diagrams. We let $S_{H_1}(r)$ denote the 2-sphere boundary of this ball.
\begin{prop}[Area and volume in $H_1$]
 $\mathrm{Area}(S_{H_1}(r))\simeq r^2$ and $\mathrm{Vol}(B_{H_1}(r))\succeq [w_1(r)]^2$.
\end{prop}

\begin{proof}
We again obtain $\mathrm{Vol}(B_{H_1}(r))\succeq [w_1(r)]^2$ from the volume of the central slab. For the area we use Lemma~\ref{lem:typei} to compute
$$
\mathrm{Area}(S_{H_1}(r))=8\mathrm{Area}(\ddii{12}{1}{r})+16r+6\simeq r^2.
$$
\end{proof}

\begin{prop}[Lower bounds for $G_0$ and $H_1$]\label{prop:lower}
$$\delta_{G_0}^{(2)}(x)\succeq x^2\qquad \delta_{H_1}^{(2)}(x)\succeq e^{\sqrt{x}}$$
\end{prop}

\begin{proof}
By Remark~\ref{rem:embedded}, (and assuming Lemmas~\ref{lem:Hvol} and ~\ref{lem:Gvol} below), it suffices now to show that $B_{G_0}(r)$ and $B_{H_1}(r)$ are embedded in $\widetilde{K_{G_0}}$ and $\widetilde{K_{H_1}}$, respectively (but see Remark~\ref{rem:composition} below).

Note that $B_{G_0}(r)$ consists of three component balls, each of which is embedded by construction: the initial equatorial slab and the two balls resulting from the two Type II moves. Call these latter two \emph{Type II balls}. If we label a vertex in the slab with the identity element $e\in G_0$, the other vertices in $B_{G_0}(r)$ naturally inherit labels by group elements. To show that $B_{G_0}(r)$ is embedded, it suffices to show that any two vertices in the cones that carry the same label must lie in the same Type II ball.

Elements $g\in G_0$ labeling vertices in a Type II ball satisfy $g=hs$, for some $h\in H_0$ and $s\in\langle{\bf u}_{0}\rangle$ (this is because $\langle{\bf u}_{0}\rangle$ acts by automorphisms). If two such elements $g_1$ and $g_2$ are equal in $G_0$, then $h_1s_1=h_2s_2$. By considering the natural retraction $G_0\to\langle{\bf u}_{0}\rangle$, we see then that $s_1=s_2$, from which it follows that $h_1=h_2$. Note that $h_1$ and $h_2$ label vertices on the embedded sphere boundary of the slab. Because distinct Type II balls are attached to this sphere along disjoint disks, $h_1=h_2$ implies that the vertices labeled by $g_1$ and $g_2$ lie in the same Type II ball.

For $B_{H_1}(r)$ the argument is similar. We think of $B_{H_1}(r)$ as having five component balls: $B_{G_0}(r)$ and four \emph{Type I balls}. As before, it suffices to show that any two vertices in the Type I balls carrying the same label must lie in the same Type I ball.

Elements $h$ labelling vertices in a Type I ball satisfy $h=gs$, where $g\in G_0$ and $h\in\langle{\bf a}_{1j}\rangle$. If two such elements $h_1$ and $h_2$ are equal in $H_1$, we deduce as before that $s_1=s_2$. In particular, the two corresponding vertices lie in Type I balls corresponding to the same stable group. Such balls are attached along disjoint disks on the embedded sphere $S_{G_0}(r)$. As before we deduce that $h_1=h_2$ only if the corresponding vertices lie in the same Type I ball.
\end{proof}

\begin{remark}\label{rem:composition}
The equivalence relation $\simeq$ behaves poorly under composition. In particular, one cannot always deduce the nature of volume as a function of area given (in)equivalencies of each as functions of $r$. It is not difficult to show, however, that these difficulties do not arise for the particular pairs of functions we encounter.
\end{remark}

\subsection{Induction steps}\label{induction}

Here we put together the pieces of the previous subsections to compute lower bounds for $\delta^{(2)}_{H_n}(x)$ and $\delta^{(2)}_{G_n}(x)$.

\medskip

\noindent {\it{Lower Bounds in $H_n$.}} Begin with a slab in $(F_2)^3$ of the form $\varphi^{w_{n-1}(r)}(a_{011})\times\varphi^{w_{n-1}(r)}(a_{021})\times y_1$, and note that the volume of this slab is $[w_n(r)]^2$. Assume by induction that a sphere $S_{G_{n-1}}(r)$ has been built upon this slab via alternating Type I and Type II moves, so that $S_{G_{n-1}}(r)$ consists of
\begin{itemize}
\item $2^{2n}$ copies of $\ddi{i}{n-1}{w_0(r)}$; and
\item $\sum_{k=1}^{2n-1}2^k$ individual 2-cells.
\end{itemize}
These hypotheses are verified above for $G_0$ ($n=1$). Let $B_{H_n}(r)$ denote the 3-ball obtained by applying a Type II move using $a_{nij}^{-r}$ to each copy of $\ddi{i}{n-1}{w_0(r)}$, and let $S_{H_n}(r)$ denote its boundary 2-sphere.  

\begin{lemma}\label{Hninduction}
The 2-sphere $S_{H_n}(r)$ consists of
\begin{itemize}
\item $2^{2n+1}$ copies of $\ddii{ij}{n}{w_0(r)}$;
\item $2^{2n+2}$ strips with dimensions $r\times 1$;
\item $\sum_{k=1}^{2n}2^k$ individual 2-cells.
\end{itemize}
Moreover, we have the following estimates:
$$
\mathrm{Vol}(B_{H_n}(r))\succeq [w_n(r)]^2\qquad\mathrm{Area}(S_{H_n}(r))\simeq [w_0(r)]^2=r^2.
$$
\end{lemma}

\begin{proof}
The volume claim is immediate from the volume of the slab. For the area, we first note that each Type I move produces four halves of some $\ddii{ij}{n}{w_0(r)}$, as pointed out in Lemma~\ref{lem:typei}. There are as many Type I moves involved in creating $S_{H_n}(r)$ as there are copies of $\ddi{i}{n-1}{w_0(r)}$ in $S_{G_{n-1}}(r)$, namely $2^{2n}$. This gives the correct number of copies of $\ddii{ij}{n}{w_0(r)}$ assuming we can show that these half-diagrams join up in pairs to form genuine copies of $\ddii{ij}{n}{w_0(r)}$. But this follows immediately from the fact that any two adjacent half-diagrams are adjacent along an edge of the form $(u_{(n-1)1})^r$ lying along the boundary of some original $\ddi{i}{n-1}{w_0(r)}$ on $S_{G_{n-1}}(r)$. See Figure~\ref{fig:wing}.

On $S_{G_{n-1}}(r)$ there is one strip in the middle of each $\ddi{i}{n-1}{w_0(r)}$, each of which is covered by a Type I move. Each Type I move gives rise to four new strips with dimensions $r\times 1$.

Finally we have the original $\sum_{k=1}^{2n-1}2^k$ individual 2-cells from $S_{G_{n-1}}(r)$, plus one more for each Type I move. Thus the number of individual 2-cells in $S_{H_n}(r)$ is $\sum_{k=1}^{2n}2^k$.

Adding up the areas of these pieces, we obtain
$$
{\text{Area}}(S_{H_n}(r))\simeq 2^{2n+1}r^2+2^{2n+2}r+\sum_{k=1}^{2n}2^k\simeq r^2.
$$
\end{proof}

\medskip

\noindent {\it{Lower Bounds in $G_n$.}} Begin with a slab in $(F_2)^3$ of the form $\varphi^{w_{n}(r)}(a_{011})\times\varphi^{w_{n}(r)}(a_{021})\times y$, and note that the volume of this slab is $[w_{n+1}(r)]^2$. Assume by induction that a sphere $S_{H_n}(r)$ has been built upon this slab via alternating Type I and Type II moves so that $S_{H_n}(r)$ consists, as shown in the previous lemma, of
\begin{itemize}
\item $2^{2n+1}$ copies of $\ddii{ij}{n}{w_{1}(r)}$;
\item $2^{2n+2}$ strips with dimensions $r\times 1$;
\item $\sum_{k=1}^{2n}2^k$ individual 2-cells.
\end{itemize}
See Figure~\ref{fig:H1} for confirmation of these hypotheses in the case $n=1$.

Let $B_{G_n}(r)$ denote the 3-ball obtained by applying a Type I move using $u_{ni}^{-r}$ to each copy of $\ddii{ij}{n}{w_{1}(r)}$, and let $S_{G_n}(r)$ denote its boundary 2-sphere.

\begin{lemma}\label{lem:Gninduction}
The 2-sphere $S_{G_n}(r)$ consists of
\begin{itemize}
\item $2^{2n+1}$ copies of $\ddi{i}{n}{w_0(r)}$;
\item $\sum_{k=1}^{2n+1}2^k$ individual 2-cells.
\end{itemize}
Moreover, we have the following estimates:
$$
\mathrm{Vol}(B_{G_n}(r))\succeq [w_{n+1}(r)]^2\qquad\mathrm{Area}(S_{G_n}(r))\simeq w_1(r).
$$
\end{lemma}

\begin{proof}
The volume claim is immediate from the volume of the slab. For the area, we first note that each Type II move produces four halves of some $\ddi{i}{n}{w_{0}(r)}$ (minus the equatorial strip), as pointed out in Lemma~\ref{lem:typeii}. There are as many Type II moves involved in creating $S_{G_n}(r)$ as there are copies of $\ddii{ij}{n-1}{w_{1}(r)}$ in $S_{H_{n}}(r)$, namely $2^{2n+1}$. This gives the correct number of copies of $\ddi{i}{n}{w_{0}(r)}$ assuming we can show that these half-diagrams join up in pairs across strips to form genuine copies of $\ddi{i}{n}{w_{0}(r)}$. But this follows immediately from the fact that any two adjacent half-diagrams are adjacent across a strip from some previous Type I move. See Figure~\ref{fig:cone}.

Finally we have the original $\sum_{k=1}^{2n}2^k$ individual 2-cells from $S_{H_{n}}(r)$, plus one more for each Type II move. Thus the number of individual 2-cells in $S_{G_n}(r)$ is $\sum_{k=1}^{2n+1}2^k$.

Adding up the area of these pieces, we obtain
$$
{\text{Area}}(S_{G_n}(r))\simeq 2^{2n+1}w_1(r)+\sum_{k=1}^{2n+1}2^k\simeq w_1(r).
$$
\end{proof}

\begin{remark}
Note that the constants involved in verifying lower bounds for the $\simeq$-class of $\delta^{(2)}(x)$ for $G_n$ and $H_n$ grow exponentially as functions of $n$. This is due to the exponential increase in the number of faces in the spheres created in these groups.
\end{remark}

The proof that these balls are embedded in their respective spaces follows exactly as in the base cases, Proposition~\ref{prop:lower}, because Type II balls and Type I balls with identical stable groups lie in distinct cosets of the stable group. Thus it remains only to show that the estimates obtained for these particular spheres suffice to establish the lower bounds (cf. Remark~\ref{rem:sparse}).

\begin{lemma}\label{lem:Hvol}
Let $S_{H_n}(r)$ denote the sphere defined above in $K_{H_n}$ with parameter $r$. 
For each $n$, there exist constants $A_1, A_2$ so that
$$
A_1r^2\leq \mathrm{Area}(S_{H_n}(r))\leq A_2 r^2.
$$
\end{lemma}

\begin{proof}
We use the combinatorial structure of $S_{H_n}(r)$ described in Lemma~\ref{Hninduction}. From the proof of Lemma~\ref{lem:deltadehn}, we have that ${\text{Area}}(\ddii{ij}{n}{w_0(r)})=2r^2$, where this is an actual equality. It follows that
$$
\mathrm{Area}(S_{H_n}(r))=2^{2n+2}r^2+2^{2n+2}r+2^{2n+1}-2.
$$
It follows (since $r\geq 1$) that we may choose $A_1=2^{2n+2}$ and $A_2=3\cdot2^{2n+2}$.
\end{proof}

\begin{lemma}\label{lem:Gvol}
Let $S_{G_n}(r)$ denote the sphere defined above in $K_{G_n}$ with parameter $r$. For each $n$, there exists a constant $C>0$ so that
$$
\text{Area}(S_{G_n}(r+1))\leq C\text{Area}(S_{G_n}(r)).
$$
\end{lemma}

\begin{proof}
The proof of Lemma~\ref{lem:thetadehn} can be refined to give the estimate
$$
w_1(r)\leq\text{Area}(\Theta_i^n(w_0(r)))\leq 3w_1(r).
$$
Combining this with the combinatorial structure of $S_{G_n}$, we then have
$$
2^{2n+1}w_1(r)+2^{2n+2}-2\leq\text{Area}(S_{G_n}(r))\leq 3\cdot 2^{2n+1}w_1(r)+2^{2n+2}-2.
$$
From this we compute
\begin{align*}
\text{Area}(S_{G_n}(r+1)) & \leq 3\cdot2^{2n+1}w_1(r+1)+2^{2n+2}-2 \\
& \leq 3\cdot 2^{2n+2}w_1(r+1) 
 =3\cdot 2^{2n+2}\frac{w_1(r+1)}{w_1(r)}w_1(r) \\
&\leq 3\cdot 2^{2n+2}\frac{w_1(r+1)}{w_1(r)}w_1(r) \\
& \leq 3\cdot 2\frac{w_1(r+1)}{w_1(r)}\text{Area}(S_{G_n}(r)).
\end{align*}
Note that the ratio $w_1(r+1)/w_1(r)$ is never larger than three ($w_1(k)$ is the length of $\varphi^k(\xi)$, and each application of $\varphi$ multiplies length by no more than three). It follows that
$$
\text{Area}(S_{G_n}(r+1))\leq 18\,\text{Area}(S_{G_n}(r))
$$
as required.
\end{proof}

We are finally able to deduce lower bounds for the 2-dimensional Dehn functions of $H_n$ and $G_n$.

\begin{prop}[Lower bounds]
For $n\geq 1$ we have $\delta_{H_n}^{(2)}(x)\succeq\exp^n(\sqrt{x})$ and $\delta_{G_n}^{(2)}(x)\succeq\exp^n(x)$.
\end{prop}

\section{Upper bounds for graphs of groups}\label{sec:upper}

In this section we show that the $2$-dimensional Dehn functions of the 
groups $G_n$ and $H_n$ are bounded above by the functions listed in 
Table~\ref{groups-table}.
In the inductive construction, each group is obtained from the 
the previous group as a multiple HNN extension.  
The upper bounds are obtained inductively using 
two general results about graphs of groups (Propositions~\ref{dehn-upper-bound} and~\ref{area-upper-bound} below). 
A version of Proposition~\ref{dehn-upper-bound} appears in \cite{pride-wang}.  We give a different proof using admissible maps.

\begin{prop}\label{dehn-upper-bound}
Let $G$ be the fundamental group of a graph of groups with the following properties:
\begin{enumerate}
\item
All the vertex groups are of type $\mathcal{F}_3$ and their $2$-dimensional Dehn functions 
are bounded above by the superadditive increasing function $f$. 
\item
All the edge groups are of  type $\mathcal{F}_2$ and their $1$-dimensional Dehn functions 
are bounded above by the superadditive increasing function $g$. 
\end{enumerate}
Then $\dehntwo{G}{x}\preceq (f \circ g)(x)$. 
\end{prop}

The upper bound obtained in the above proposition may not always be the optimal one.  Proposition~\ref{area-upper-bound} below gives another way to obtain an upper bound.  It relies on the existence of a function 
bounding the area-distortions of the edge groups in $G$.

\begin{prop}\label{area-upper-bound}
Let $G$ be the fundamental group of a graph of groups with the following properties:
\begin{enumerate}
\item
All the vertex groups are of type $\mathcal{F}_3$ and their $2$-dimensional Dehn functions 
are bounded above by the superadditive increasing function $f$. 
\item
All the edge groups are of  type $\mathcal{F}_2$.
\item 
There exists a function $h$ such that if $w$ is a word representing $1$ in 
an edge group $\G$, then $\area{\G}{w} \leq h(\area{G}{w})$.
\end{enumerate}
Then $\dehntwo{G}{x} \preceq f(xh(x))$.
\end{prop}

Sections~\ref{sec:labels} and ~\ref{sec:transverse} contain some preliminaries for the proofs of these propositions.  The proofs are contained in Sections~\ref{sec:strategy}-\ref{subsec:fill-vol}.  The propositions are used to prove upper bounds for $\dehntwo{G_n}x$ and $\dehntwo{H_n}x$ in Section~\ref{subsec:inductive-upper}.

\subsection{The $3$-complex $\wide{K_G}$}\label{sec:labels}
We start with the standard total space $K_G$ of the graph of spaces associated with $G$.
Recall that ${K_G}$  
is a 
quotient map 
\begin{equation}\label{eq:quotient-map}
q:\sqcup  {K_v}\, \sqcup \,(\sqcup {K_e} \times [0,1]) \to {K_G},
\end{equation}
where the $K_v$ are indexed by the vertex set of the graph and the $K_e$ are indexed by its edge set.  
Note that $q$ identifies each $K_e \times i$ (where $i=0,1$) with its image in a vertex space $K_v$ under a map induced by the edge inclusion $G_e \to G_v$. 
 
The vertices of $\wide{K_G}$ are exactly the vertices of the $\wide{K_v}$.  For $1 \leq i \leq 3$, an $i$-cell of $\wide{K_G}$ is either an $i$-cell of some $\wide{K_v}$ or of the form $c \times [0,1]$, where $c$ is an $(i-1)$-cell of some $\wide{K_e}$.  

\medskip
\noindent{{\bf Labeling on $\wide{K_G}$}.}
Each $1$-cell of $\wide{K_G}$ of the first type mentioned above is labeled by a generator of the corresponding vertex group.  All $1$-cells of the second type are labeled by the letter $e$.  

The boundary of a $2$-cell of the first type is labeled by a relation in the corresponding vertex group. A $2$-cell of the second type is homeomorphic to $E \times [0,1]$, where $E$ is an edge of some $\wide{K_e}$. Then $E$ corresponds a generator, say $g$, of the edge group $G_e$.  Let $G_u$ and $G_v$ be the vertex groups adjacent to $G_e$, and let $X$ and $Y$ be words in the generators of $G_u$ and $G_v$ respectively, which represent the images of $g$ under the edge inclusion maps.  
Then the $2$-cell $E \times [0,1]$ has boundary label $eXe^{-1}Y^{-1}$.  
Note that any boundary label that contains $e$ is of this form.

  Within each piece of $\wide{K_G}$ of the form $\wide{K_e} \times [0,1]$, we identify the slice $\wide{K_e} \times 1/2$ with $\wide{K_e}$.  
Every cell $c \times [0,1]$ in $\wide{K_e} \times [0,1]$ contains an embedded copy $c \times 1/2$ of a cell of $K_e$. If $c$ is a $1$-cell, we label 
$c \times 1/2$ by the generator of the edge group $G_e$ that labels $c$.
We use this additional cell structure and labeling in the course of the construction below (to define central words of annuli), but we do not think of these as cells of $\wide{K_G}$.

\subsection{Transverse maps}\label{sec:transverse}
In proving upper bounds for $1$-dimensional Dehn functions, one often uses the notion of an $e$-corridor or $e$-annulus in a van Kampen diagram.  To facilitate the definition of an analogous object in the present setting, we require maps $f:S^2\to \wide{K_G}$ to be \emph{transverse}.   Transversality, a condition more stringent than admissibility, gives rise to a \emph{generalized handle decomposition} of $S^2$.  We summarize a few essential facts here and refer 
to~\cite{bf} for more details. 

An \emph{index $i$ handle} of dimension $n$ is a product $\Sigma^i \times D^{n-i}$,
where $\Sigma^i$ is a compact, connected $i$-dimensional manifold with boundary, and $D^{n-i}$ is a closed disk.  A \emph{generalized handle decomposition} of an 
$n$-dimensional manifold 
$M$ is a filtration $\emptyset=M^{(-1)} \subset M^{(0)} \subset \cdots \subset M^{(n)}=M$
by codimension-zero submanifolds,
such that for each $i$, $M^{(i )}$ is obtained from $M^{(i-1)}$ by attaching finitely many index $i$ handles.  Each $i$-handle $\Sigma^i \times D^{n-i}$ is attached via an embedding $\partial \Sigma^i \times D^{n-i} \to M^{(i-1)}$.
A map 
$f$ from $M$ to
a CW complex $X$ is \emph{transverse} to the cell structure of 
$X$ if $M$ has a generalized handle decomposition such
that the restriction of $f$ to each handle is given by projection onto the second factor,
followed by the characteristic map of a cell of $X$. 
We say
that $X$ is a \emph{transverse CW complex} if the attaching map of every cell is transverse to the
cell structure of the skeleton to which it is attached.
We will need the following result of Buoncristiano, Rourke, and Sanderson.

\begin{theorem}[\cite{brsand}] \label{thm:transverse}
Let $M$ be a compact smooth manifold and $f\!:M\to X$ a continuous map into a transverse CW-complex. Suppose $f|_{\partial M}$ is transverse. Then $f$ is homotopic rel $\partial M$ to a transverse map $g:M\to X$.
\end{theorem}

In order to apply this theorem, one needs a transverse complex.  
The complex $K_G$ defined above can be made transverse by inductively applying the theorem to attaching maps of cells.  This procedure can be done in way that preserves 
the homeomorphism type of the complex and its partition into cells. (See Section 3 
of~\cite{bf} for the details of this procedure.)

Lastly, if one applies the theorem to an admissible map to make it transverse, its combinatorial volume does not change.  (This is because the preimages of $n$-cells do not change during the course of the homotopy, except possibly by shrinking slightly)

\medskip
\noindent{\bf Existence of $e$-annuli.}
We are concerned with transverse maps $\tau : S^2 \to \wide{K_G}$.  Such a map induces a  decomposition of $S^2$ into handles of index $0$, $1$, and $2$, and each $i$-handle maps to a $(2-i)$-dimensional cell of $\wide{K_G}$. 
Each $1$-handle inherits the label of the $1$-cell that it maps onto.  If there is a $1$-handle labeled $e$,
there are two possibilities:
either the $1$-handle is of the form $S^1 \times I$ (an annulus) or it is adjacent to a $0$-handle.  In the latter case, the
$0$-handle is mapped homeomorphically to a $2$-cell in $K_G$ whose boundary is labeled by
a word of the form $eXe^{-1}Y^{-1}$, and is therefore adjacent to exactly one other $1$-handle labeled $e$.  
This new $1$-handle is, in turn, adjacent to another $0$-handle.  Continuing this process, we obtain a finite concatenation of alternating $1$-handles (labeled $e$) and $0$-handles, which 
forms an annulus.
In either case, we call the annulus an \emph{$e$-annulus}.  Note that an $e$-annulus can consist of a single $1$-handle of the form $S^1 \times I$.  

The only handles adjacent to an $e$-annulus are $1$-handles 
that are not labeled $e$ and $2$-handles.  Thus $e$-annuli never intersect each other.  (In particular, such an annulus never intersects itself.)  

\medskip
\noindent{\bf The central and boundary words of an $e$-annulus.}
Any $0$-handle in an $e$-annulus is mapped homeomorphically to a $2$-cell of the form $E \times [0,1]$, 
where $E$ is an edge in an edge space, say $K_{\epsilon}$, 
and $E \times 1/2$ is labeled by a generator of 
the edge group $G_{\epsilon}$. 
Concatenating such labels from successive $0$-handles, we obtain the \emph{central word} of the $e$-annulus. We think of the central word as labeling a \emph{central circle} $C$ that runs through the interior of the annulus. 
The circle $C$ is simply the union of the segments  
$[0,1] \times \{1/2\}$ through the center of each $1$-handle of the annulus and the curves $\tau^{-1}(E \times 1/2)$ in each $0$-handle.  
Note that $\tau(C) \subset K_{\epsilon} \times 1/2$.

The boundary of the $2$-cell is labeled by $eXe^{-1}Y^{-1}$, where
$X$ and $Y$ are words in the generators of the two adjacent vertex groups.  Concatenating such words from successive $0$-handles in the $e$-annulus gives the two \emph{boundary words} of the annulus.

\subsection{General strategy for the proofs}\label{sec:strategy}
We assume that $\wide{K_G}$ has been made transverse by applying the procedure mentioned after Theorem~\ref{thm:transverse}.  
  Given any admissible map $\sigma: S^2 \to \widetilde{K_G}^{(2)}$ with $\mathrm{Area}(\sigma) =x$, we construct an admissible filling $\bar \sigma: D^3 \ra \widetilde{K_G}$ whose volume is bounded above by a function $\simeq$-equivalent to $f( g (x))$ (for Proposition~\ref{dehn-upper-bound}) or $f(xh(x))$ (for Proposition~\ref{area-upper-bound}).

To construct the admissible filling $\bar \sigma$,
we think of $D^3$ as a cone $S^2 \times [0,1]/ S^2\times 1$ and decompose it into two pieces, $A=S^2 \times [0,1/2]$ and a closed ball $B=S^2 \times [1/2, 1]/ S^2 \times 1$.
It follows from Theorem~\ref{thm:transverse} that
there is a homotopy $\Psi: S^{2} \times [0,1/2]\to \wide{K_{G}}^{(2)}$ such that $\Psi_{0}=\sigma$, and $\Psi_{1/2}= \tau$ is a transverse map with $\mathrm{Area}(\tau) = \mathrm{Area}(\sigma)$. 
In Section~\ref{subsec:filling} we construct an admissible filling $\bar \tau: B \to \wide{K_G}$. 
Then 
\begin{equation}\label{eq:sigmabar}
\bar \sigma =\begin{cases} \Psi & \text{ on }A \\
\bar \tau& \text{ on }B\end{cases}
\end{equation}
(See Figure~\ref{fig:fill}.)  Note that $\bar \sigma$ is admissible and that $\vol{\bar \sigma}=\vol{\bar \tau}$.

To construct the filling $\bar \tau$ we decompose the ball $B$ into two families of balls with disjoint interiors.  The first family consists of 
``slabs'' homeomorphic to $D^2 \times I$.  The lateral boundary $\partial D^2 \times I$ of each slab lies on $\partial B$. The space formed by deleting the interiors and lateral boundaries (i.e.~$D^2 \times (0,1)$) of the slabs is a disjoint union of balls, which make up the second family.  We construct admissible maps on each of the component balls in the two families in such a way that they agree on common boundaries and agree with $\tau$ on $\partial B$.

Under the standard map from $\wide{K_G}$ to the associated Bass-Serre tree, 
the image in $\wide{K_G}$ of a slab from the first family above maps to an edge of the tree, while the image of a ball from the second family maps to a vertex of the tree. 
The volume of 
a slab is related to the areas of its boundary disks $D^2 \times 0$ and 
$D^2 \times 1$.   
The volume of a ball of the second type is controlled using the $2$-dimensional Dehn function of the corresponding vertex group.  The details of the volume estimates are given in Section~\ref{subsec:fill-vol}.

\begin{figure}
\psfrag{A}{{\small $A$}}
\psfrag{A1}{{\small $A_1$}}
\psfrag{B}{{\small $B$}}
\psfrag{dB}{{\small $\partial B$}}
\psfrag{t}{{\small $\bar{\tau}$}}
\psfrag{D}{{\small $D^2\times [0,1]$}}
\psfrag{ps}{{\small $\Psi$}}
\begin{center}
\includegraphics[width=110mm]{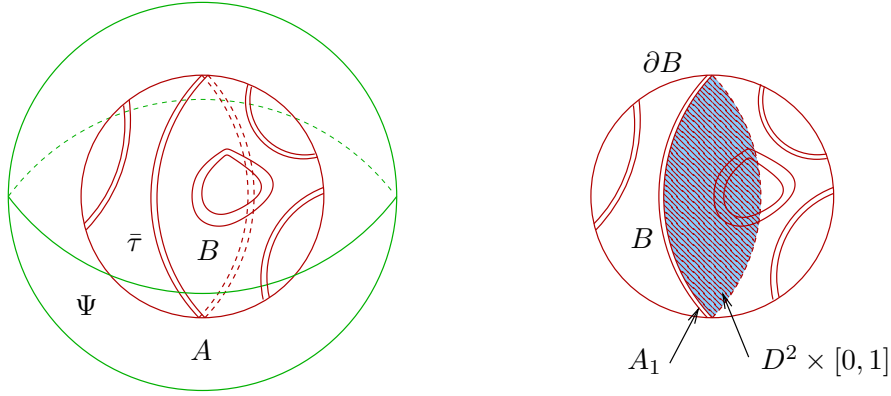}
\end{center}
\caption{Filling the Admissible map $\sigma: S^2 \to \widetilde{K_G}$. \label{fig:fill}}
\end{figure}

\subsection{The admissible filling $\bar \tau$} \label{subsec:filling} 
An admissible filling of the transverse map $\tau: \partial B \to \wide{K_G}^{(2)}$ is constructed as follows.
If $\partial B$ has no $e$-annuli, the image of $\tau$ lies completely in one of the ($3$-dimensional) vertex spaces $K_v$.  Since $f$ is an upper bound for the $2$-dimensional Dehn functions of the vertex groups (in both propositions), there exists an admissible filling $\bar \tau:B \to \wide{K_v}$, with $\vol{\bar \tau} \preceq f(\area{}{\tau})$.

If $\partial B$ has $e$-annuli, then $\bar \tau$ is constructed inductively. 
The induction is on the number of $e$-annuli.  
If $A_1$ is an $e$-annulus, form a new space $B_1$, by 
gluing a slab $D^2 \times [0,1]$ to $\partial B$ along $A_1$, i.e.  
\[
B_1 = \partial B \sqcup (D^2 \times [0,1]) \; /\;    A_1 \sim (\partial D^2 \times [0,1]).
\]

If $A_1$ consists of a single $1$-handle, then it is homeomorphic to $S^1 \times [0,1]$. With this identification, the restriction of $\tau$ to $A_1$ is simply projection onto the second factor, followed by the characteristic map of a $1$-cell. 
Then $\tau$ extends to a map on $D^2 \times [0,1]$ whose image is the same $1$-cell. (The map is constant on each disk $D^2 \times t$.)  
This defines an admissible map $\tau_1 :B_1 \to \wide{K_G}$

Otherwise let $w_1$ and $C_1$ be the central word and central circle of $A_1$. Now $\tau|_{C_1}$ is an admissible map into some $\wide{K_{\epsilon}}$.  The circle $C_1$ divides $\partial B$ into two components, and the restriction of $\tau$ to the closure of one of these components is  
a filling of $\tau|_{C_1}$.  Thus $w_1$ represents the identity in $G$ (and hence in the edge group $G_{\epsilon}$).  It follows that $w_1$ can be filled in $\wide{K_{\epsilon}}$, i.e., 
there is an admissible map $D^2 \to \wide{K_{\epsilon}}$ that agrees with $\tau|_{C_1}$ on $\partial D^2$. 
Extend this to a map $\eta: D^2 \times [0,1]\to \wide{K_{\epsilon}} \times [0,1]$ (by defining it to be the identity on the second factor).  
Restrict the quotient map $q$ from Section~\ref{sec:labels} 
to $K_{\epsilon} \times [0,1]$ and let $\tilde q$ denote the unique lift to universal covers so that $\tilde q \circ \eta$ agrees with $\tau$ on $A_1$.

Now define $\tau_1: B_1 \to \wide K_G$ by 
\[ 
 \tau_1 = \begin{cases} \tau  \text{ on } \partial B \\
 						\tilde q \circ \eta \text{ on } D^2 \times I
\end{cases}.
\]
Note that $\tau_1$ is an admissible map of $B_1$ into $\wide{K_G}$
(i.e. $\tau_1$ restricted to a component of the inverse image of an open $3$-cell is a homeomorphism). Moreover, the restrictions of $\tau_1$ to the boundary disks $D^2 \times i$, where $i=0,1$, are admissible maps into $\wide{K_G}^{(2)}$. (In fact the boundary disks are mapped into the $2$-skeletons of vertex spaces.)

The complement $B_1 \setminus D^2 \times (0,1)$ is a disjoint union of two spheres, which we call the \emph{complementary} spheres.  The restrictions of $\tau_1$ to the complementary spheres are admissible maps into $\wide{K_G}^{(2)}$. 
Note that even though $\tau$ was a transverse map, we do not require $\tau_1$ to be transverse.  The transversality of $\tau$ was only used to conclude that the
$e$-annuli on $\partial B$ are all well-defined, embedded, and mutually 
disjoint.  This is clearly the case for the restrictions of $\tau_1$ to each of
the complementary spheres.  Note that,  
since the boundary disks $D^2 \times i$ do not contain any $2$-cells labeled $e$, no new $e$-annuli are added in the course of the above procedure.
Now we repeat the procedure with $\tau$ replaced by each of the restrictions of $\tau_1$ to the complementary spheres.

Let $N$ be the total number of $e$-annuli on $\partial B$.  By induction we obtain
an admissible map $\tau_N:B_N \to \wide{K_G}$, and 
complementary spheres $S_1, \dots S_{N+1}$ (since in each step one of the complementary spheres is divided into two).  
Let $\tau_{Ni}$ denote the restriction of $\tau_N$ to $S_i$.
Then $\tau_{Ni}$ is an admissible map into the $2$-skeleton of some vertex space.  
 
As before, for each $i$, there is an admissible extension $\bar \tau_{Ni}$ of $\tau_{Ni}$ defined on $D^3_i$, a $3$-dimensional ball with boundary $S_i$, 
such 
that $\vol{\bar \tau_{Ni}} \preceq f(\area{}{\tau_{Ni}})$. 
Note that the ball $\displaystyle B= \cup_{i=1}^{N+1} D^3_i \cup B_N$, and 
define 
$\bar \tau: B \to \wide{K_G}$ by
$$ \bar \tau = \begin{cases} 
				\tau_N \text{ on } B_N \\
				\bar \tau_{Ni} \text{ on } D^3_i, \text{ for } 1 \leq i \leq N+1
\end{cases}
$$
Note that $\bar \tau$ is an admissible extension of $\tau$.  Finally, the admissible filling $\bar \sigma$ of the original map $\sigma$ is defined as in Equation~\ref{eq:sigmabar}.

\subsection{Filling volumes in graphs of groups}\label{subsec:fill-vol}
Let $A_1, \dots, A_N$, $w_1, \dots, w_N$ and 
$\wide{K_{e_1}}, \dots, \wide{K_{e_N}}$ denote the $e$-annuli on $\partial B$, their central words and the corresponding edge spaces respectively.  
For any $A_i$, the restriction of $\tau_N$ to the slab that it bounds has volume equal to $\area{K_{e_i}}{w_i}$. (If $A_i$ is an $e$-annulus of the form $S^1 \times [0,1]$ for some $i$, then $w_i$ is empty and 
$\area{K_{e_i}}{w_i}=0$.)

There exists a constant $M$, which depends only on $G$, such that if $h$ is a boundary word of one of the $e$-annuli, say $A_i$, (so that $h$ is a word in the generators of a vertex group, say $G_v$), then $\area{K_v}{h}\leq M \area{K_{e_i}}{w_i}$. This ensures that at the $i$th step of the procedure, (i.e., when a slab is glued in along $A_i$ to form $B_i$), the \emph{total}
area of the complementary spheres increases by at most $2M \area{K_{e_i}}{w_i}$.
Using the superadditivity of $f$, and the fact that the
 fillings $\bar \tau_{Ni}$ of the complementary spheres were chosen so that 
 $\vol{\bar \tau_{Ni}} \leq f(\area{}\tau_{Ni})$, 
we have 
\begin{equation}\label{volume-estimate}
\begin{aligned}
\vol {\bar \sigma} & = \vol{\bar \tau}\\
&= \text{ Total volume of slabs } + \text{ Total 
volume of balls }\\
&\leq  \sum_{i=1}^N \area{K_{e_i}}{w_{i}} + \sum_{i=1}^{N+1} f(\area{}{\tau_{Ni}})\\
&\leq \sum_{i=1}^N \area{K_{e_i}}{w_i} + f\left(\sum_{i=1}^{N+1} \area{}{\tau_{Ni}}\right)\\
&\leq \sum _{i=1}^N\area{K_{e_i}}{w_i} + f\left(\area{}{\tau}+ 2M\sum_{i=1}^N \area{K_{e_i}}{w_{i}}\right) 
\end{aligned}
\end{equation}

We are now ready to complete the proofs of Propositions~\ref{dehn-upper-bound} and~\ref{area-upper-bound}.  
We retain the notation developed above. 
 
\begin{proof}[Proof of Proposition~\ref{dehn-upper-bound}] 
Let $\sigma: S^2 \ra \wide{K_G}^{(2)}$ be an admissible map with area $x$.  We obtain an admissible 
filling $\bar \sigma: D^3 \ra \wide{K_G}$ as described in the procedure above. 
Equation \eqref{volume-estimate} estimates the volume of this filling. 
By the assumption on the $1$-dimensional Dehn functions of the edge groups,
we have that $\area{K_{e_i}}{w_{A_i}} \leq g(|w_{A_i}|)$.
So by the superadditivity of $g$, we have $\sum \area{K_{e_i}}{w_{A_i}} \leq g(\sum |w_{A_i}|) \leq 
g(x)$,
since the total lengths of central words of annuli cannot be more than the total area of $\sigma$. 
So the estimate for the volume is
$
\vol{\bar \sigma} \leq g(x) +f(x+2Mg(x)).
$
This gives $$\dehntwo{G}{x} \preceq g(x) +f(x+2M g(x)).$$
However, since $x \preceq f(x)$, and $x \preceq g(x)$, we have 
$g(x)+f(x+2Mg(x)) \preceq f(g(x))$.  This is the required upper bound. 
\end{proof}

\begin{proof}[Proof of Proposition~\ref{area-upper-bound}]
As in the previous proof, let $\sigma: S^2 \ra \wide{K_G}^{(2)}$ be an admissible map 
with area $x$
and let $\bar \sigma$ be the admissible filling from the procedure above.  This time, we have the 
condition $\area{K_{e_i}}{w_{A_i}} \leq h(\area{K_G}{w_{A_i}})$ for each $i$. 
The restriction of $\tau$ to a ``hemispherical'' piece of $\partial B$ gives a filling of $w_{A_i}$
of size at most $x$.  
Since there are at most $x$ annuli, 
the estimate in Equation~\ref{volume-estimate} becomes
\begin{align*}
\vol{\bar \sigma} &\leq \sum_{i+1}^N \area{K_{e_i}}{w_{A_i}} + 
f\left(x+ 2M\sum_{i+1}^N \area{K_{e_i}}{w_{A_i}}\right) \\
&\leq xh(x) + f(x + 2M xh(x)).
\end{align*}
This give $\dehntwo{G}{x} \preceq f(xh(x))$.
\end{proof}

\subsection{Upper bounds for the super-exponential examples} 
\label{subsec:inductive-upper}
The proof of the upper bounds for the super-exponential examples is by induction.  We have an alternating sequence of groups
\[
H_{0} < G_0 < H_1 < G_1 < H_2 < G_2 \cdots
\]
where $H_{0}=F_2 \times F_2$.  
(Note that this is different from the $H_0$ defined in Table~\ref{groups-table}.  In Section~\ref{sec:group-definitions}, example~\ref{ex:G0}, the group $G_0$ was obtained from $F_2\times 
F_2 \times F_2$ by coning over $F_2 \times F_2$.  This can also be viewed as a multiple
($4$-fold) HNN extension with base $F_2 \times F_2=H_{0}$, where two stable letters act via the identity and two act via $\varphi \times \varphi$.)

All of the above groups are of type $\mathcal{F}_3$. Apart from $H_{0}$, they all have $3$-dimensional $K(\pi,1)$'s.
As described in Section~\ref{sec:inductive-group-defs}, each group in the sequence is a multiple HNN extension of the previous one.

\begin{proof}[Proof of upper bounds for $\delta^{(2)}_{H_n}$ and $\delta^{(2)}_{G_n}$]
Since $F_2 \times F_2$ has a $2$-dimensional $K(\pi, 1)$ we have 
$\dehntwo{H_{0}}{x} = \dehntwo{F_2 \times F_2}{x} \simeq x$.  
This starts the induction.  

\smallskip

{\it Step 1. Deducing $\dehntwo{G_n}{x}$ upper bounds from $\dehntwo{H_{n}}{x}$ upper bounds ($n\geq 0$).}
This is a straightforward application of Proposition~\ref{dehn-upper-bound}.  The group $G_n$ is the fundamental group of a graph of groups where the underlying graph is a bouquet of two circles (four if $n=0$), the vertex group is $H_{n}$ and the edge groups are $F_{2^{n+1}} \times F_2$.  Further,

\begin{enumerate}
\item $\dehntwo{H_{0}}{x} \simeq x$ (base case) and 
$\dehntwo{H_{n}}{x}\simeq \exp^{n}(\sqrt x)$ for $n>0$ (induction hypothesis).

\smallskip

\item  $\dehnone{F_{2^{n+1}} \times F_2}x \simeq x^2$.
\end{enumerate}
Proposition~\ref{dehn-upper-bound} now implies that $\dehntwo{G_0}{x}\preceq x^2$ and
$\dehntwo{G_n}x \preceq \exp^n(x)$ for $n>0$.

\smallskip

{\it Step 2. Deducing $\dehntwo{H_n}x$ upper bounds from $\dehntwo{G_{n-1}}x$ upper bounds ($n > 0$). }  
This is more involved than the previous step.  The group $H_n$ is the fundamental group of a graph of groups where the underlying graph is a 
wedge of $2^{n+1}$ circles, the vertex group is $G_{n-1}$ and the edge groups are all isomorphic to
$F_2 \rtimes F_4$.  The result will follow from Proposition~\ref{area-upper-bound} together with 
the following lemma, which describes how areas in the edge groups get distorted in $H_n$.  

\begin{lemma}[Area-distortion of $\G$ in $H_n$]\label{lem:area-distortion}
Let $\G$ be an edge group in the graph of groups description of $H_n$.  
Then there exists a constant $\beta_n$, which depends only on $n$, such that for any word $w$ in the generators of $\Gamma$ that represents $1$, 
\begin{equation}\label{eq:area-distortion}
\area{\G}{w} \leq (\beta_n \area{H_n}{w}e^{\sqrt{\beta_n\area{H_n}{w}}})^2.
\end{equation}
\end{lemma}

\noindent
Section~\ref{sec:area-distortion} is devoted to the proof of this lemma.  We now have:
\begin{enumerate}
\item $\dehntwo{G_0}x \preceq x^2$ (Step 1) and $\dehntwo{G_{n-1}}{x} \preceq \exp^{n-1}(x)$ 
for $n>1$ (induction hypothesis)

\smallskip

\item
By Lemma~\ref{lem:area-distortion}, the function $h(x)=(\beta_n xe^{\sqrt{ \beta_n  x}})^2$ satisfies the third condition in the statement of Proposition~\ref{area-upper-bound}.
\end{enumerate}
Proposition~\ref{area-upper-bound} and the fact that $h(x)\simeq e^{\sqrt x}$ now imply that 
$\dehntwo{H_1}x \preceq {(xe^{\sqrt x})}^2$ and
$\dehntwo{H_n}x \preceq \exp^{n-1}(xe^{\sqrt x})$ for $n>1$. 
Since $x e^{\sqrt x}  \simeq e^{\sqrt x} $ and $({e^{\sqrt x})}^2 \simeq e^{\sqrt x}$, we have 
$\dehntwo{H_n}x \preceq \exp^{n}(\sqrt x)$ for all $n\geq 1$. 

This completes the proof of the upper bounds for the super-exponential examples.  
\end{proof}

\section{Proof of the area-distortion lemma}\label{sec:area-distortion}
This section provides a detailed proof of Lemma~\ref{lem:area-distortion}.  
Let $\Gamma$ be any of the $2^n$ groups isomorphic to $F_2 \rtimes_\theta F_4$ listed in the row corresponding to $H_n$ in Table~\ref{groups-table}.   
In this section we will use the following notation for $\Gamma$:
$$
\Gamma = \langle {\bf a} \rangle \rtimes_\theta \langle {{\bf u}}, {\bf b} \rangle
$$ 
Here ${\bf u}= {\bf u}_{n-1}$. (Note that each of the $2^n$ edge groups has $\langle {\bf u}_{n-1} \rangle$ as a subgroup.) Either ${\bf a} = \stable{(n-1)}i$ with ${\bf b} = \mathcal{L}(i)$ or ${\bf a} = {\bf u}_{n-2}^{-1} \stable{(n-1)}i$ with ${\bf b} = \mathcal{L}(i+2^{n-1})$ for some $1 \leq i \leq 2^{n-1}$. 
The $F_4$ generated by ${\bf u}$ and ${ \bf b}$ acts on $\langle {\bf a} \rangle$ via $\theta$, as defined in Definition~\ref{def:suspend}.

As we are concerned only with areas in this section, we work with van Kampen diagrams, rather than admissible or transverse maps, and combinatorial complexes rather than transverse ones.   
Since we are working with fixed presentations, we will use the phrase ``van Kampen diagram over $G$'' to mean ``van Kampen diagram over the fixed presentation for $G$.'' 

\subsection{Strategy} Let $w$ be a word in the generators of $\Gamma$ that represents $1$.  Then there exists a van Kampen diagram $D$ for $w$ over $H_n$, which is area-minimizing, i.e. $\area{}{D} = \area{H_n}{w}$. 
First realize $D$ as a union of ${\bf u}$-corridors and complementary
regions.  Next, use this structure to produce a van Kampen diagram $\Delta$ for $w$ over $\Gamma$, which has the same combinatorial decomposition into ${\bf u}$-corridors and complementary regions. 
Finally show that $\area{}{\Delta}$, and hence $\area{\Gamma}{w}$, is bounded above by the quantity on the right hand side of inequality~(\ref{eq:area-distortion}).

\begin{remark}\label{rmk:nonsing}
For the rest of this section, we assume that the van Kampen diagram $D$ has the property that every edge on the boundary of $D$ belongs to a $2$-cell in $D$.  We can restrict to this case using the superadditivity of the function
$(Cx 3^{\sqrt{Cx}})^2$.
\end{remark}

\begin{remark}
Throughout this proof we abuse notation and use 
boldfaced letters to denote either a pair of generators for a free group or a single one of these generators; it will be clear from the context which of these we mean.  For example, ``${\bf u}$-corridor'' is used to mean $ u_{(n-1)i}$-corridor, where $i$ is $1$ or $2$, as we do not need to distinguish between these.    
When we refer to, say, ``the word ${\bf u} x {\bf u}^{-1}$'' it is understood that both instances of ${\bf u}$ in the word refer to the same generator (either $u_{(n-1)1}$ or 
$ u_{(n-1)2}$). 
\end{remark}

\subsection{Geometry of ${\bf u}$-corridors}\label{subsec:corridor-geom}
The following is a complete list of the relations involving the generators ${\bf u} = {\bf u}_{n-1}$ in the presentation for $H_n$. 
\begin{enumerate}
\item 
Let $
E= \langle \stable{(n-1)}{1}, \dots , \stable{(n-1)}{2^{n-1}},
{\bf u}_{n-2} \rangle \simeq F_{2^n}\times F_2
$.  Recall that $G_{n-1}$ is the cone of $H_{n-1}$ over $E$  
with stable letters ${\bf u}$ and relations
\begin{equation*}
{\bf u}\,g \,{\bf u}^{-1}\, (\varphi(g))^{-1} =1 \quad
\text{where } g=\text{any generator for } E.
\end{equation*}
 
\item Recall that $H_n$ is a multiple HNN-extension of $G_{n-1}$ 
with stable letters 
$\stable nj$, with $1 \leq j \leq 2^n$.  The new relations involving ${\bf u}_{(n-1)}$ are the commuting relations:
\begin{equation*}
{\bf u}\,{\bf a}_{nj} \,{\bf u}^{-1}\, {\bf a}_{nj}^{-1}=1 \quad \text{ for } 
1 \leq j \leq 2^n.
\end{equation*}  
\end{enumerate}
Since these are all the relations involving the ${\bf u}$, it makes sense to talk about \emph{${\bf u}$-corridors} and \emph{${\bf u}$-annuli} in van Kampen diagrams over $H_n$.
The reader may refer to Section 7.2 of \cite{bridson} for the definitions and properties of corridors and annuli (called rings in \cite{bridson}).
By the assumption on $D$ in Remark~\ref{rmk:nonsing},
every edge labeled ${\bf u}$ in $\partial D$ is part of a non-trivial ${\bf u}$-corridor in $D$. Although $D$ may contain ${\bf u}$-annuli, only ${\bf u}$-corridors will play a key role 
in the argument below.

\medskip
\noindent {\bf Area-length inequality for ${\bf u}$-corridors in $D$ and $\Delta$.} 
The boundary of a ${\bf u}$-corridor $\mathcal C$ over $H_n$ or $\Gamma$ is labeled by a word
of the form ${\bf u} X_1 {\bf u} ^{-1} X_2$. 
We call the $X_i$  the \emph{horizontal boundary words} of $\mathcal C$. From the presentations of $H_n$ and $\Gamma$, we see that a horizontal boundary of a single $2$-cell involving ${\bf u}$ has length at most $3$.  
(For $\Gamma$ we measure lengths in the intrinsic metric, which may differ from the inherited one by a factor of $2$.) Thus
\begin{equation}\label{eq:corridor-area} 
\frac{|X_i|}{3} \leq \area{}{\mathcal{C}} \leq |X_i|, \quad \text{for } i=1,2.
\end{equation}

\medskip
\noindent {\bf Exponential distortion of corridors.}
The following lemma produces a ${\bf u}$-corridor over $\Gamma$ corresponding to a given 
${\bf u}$-corridor in $D$, and relates their lengths. 
\begin{lemma}[Exponential distortion of corridors]\label{lem:corr-dist}
Let $\mathcal C_D$ be a ${\bf u}$-corridor in $D$ with boundary label ${\bf u} X_1 {\bf u} ^{-1} X_2$. Then there exists a 
${\bf u}$-corridor $\mathcal C_{\Gamma}$ over $\Gamma$, with boundary label 
${\bf u} Y_1 {\bf u} ^{-1} Y_2$ such that the $Y_i$ are words in ${\bf a}$ with
$Y_i =_{H_n} X_i$ for $i=1,2$.  Furthermore, 
\begin{equation} \label{eq:length-comp}
|Y_i| \;\leq \; 3 \,(3^{3 |X_i|}).
\end{equation}
\end{lemma}
 
\begin{proof}
Let $\Gamma'= \langle {\bf a}'\rangle \rtimes \langle {\bf u} , {\bf b}'\rangle$ 
be a group from the list  
of edge groups for $H_n$ (possibly different from $\Gamma$), and let $E$ be as defined in the beginning of subsection~\ref{subsec:corridor-geom}. The following two properties will be used repeatedly to construct $\mathcal{C}_{\Gamma}$. 

\medskip
\noindent{\bf Property 1.}
$E \cap \Gamma' = \langle {\bf a}' \rangle$
\begin{proof}
It is easy to see (using the normal form for semidirect products) that 
this is equivalent to the statement $E \cap \langle {\bf u} , {\bf b}' \rangle=1$. The latter statement follows easily from an elementary argument using 
${\bf u}$-corridors. 
\end{proof}

\medskip
\noindent{\bf Property 2.}
$\langle {\bf a}'\rangle$ is a retract of $E$. 
\begin{proof}
If $\langle {\bf a}' \rangle$ is generated by $\stable {(n-1)}i$ for some 
$i$, then the retraction is simply the projection of $E$ onto $\langle {\bf a}' \rangle$.  If it is a subgroup generated by ${\bf u}_{n-1}^{-1} \stable{(n-1)}i$ for some $i$, then observe that the following map is a retraction:
\[
\stable {(n-1)}j \mapsto 1 \; (j\neq i)\; ; \quad \stable {(n-1)}i\mapsto {\bf u}_{n-1}^{-1} \stable{(n-1)}i \; ;\quad {\bf u}_{n-1} \mapsto 1.
\]
\end{proof}

Since the horizontal boundary $X_1$ of $\mathcal C_D$ has endpoints on the boundary of $D$, there is a subword $W$ of $w$ such that 
$W^{-1}X_1$ represents $1$.  By Britton's Lemma applied to the multiple HNN extension description of $H_n$ in Table~\ref{groups-table}, we conclude that $X_1$ must have an innermost subword of the form ${\stable ni}v\stable ni^{-1}$ or ${\stable ni}^{-1}v\stable ni$ for some $i$. Here $v$ is a word in the generators of $E$ that represents an element of 
the edge group corresponding to $\stable ni$, say $\Gamma'=\langle {\bf a}'\rangle \rtimes \langle {\bf u} , {\bf b}'\rangle$.  
By Property 1, there exists a (reduced) word $v'$ in the generators ${\bf a}'$ representing the same element as $v$.  By Property 2
 we have $|v'| \leq |v|$ (since retracts are length-non-increasing).  Then if the innermost subword is of the form 
 ${\stable ni}v\stable ni^{-1}$, it can be replaced with $\varphi(v')$, and if it is of the form 
 ${\stable ni}^{-1}v\stable ni$, it can be replaced with $\varphi^{-1}(v')$.   
Both $\varphi(v')$ and $\varphi^{-1}(v')$ have length at most $3|v|$.

Repeat this procedure (at most $|X_1|/2$ times) until all instances of $\stable ni$ have been eliminated.  We obtain a word $X_1'$ in 
the generators of $E$ of length 
$|X_1'| \leq 3^{|X_1|/2}\leq 3^{|X_1|}$.  

Since $X_1'$ represents the same word as $W$, which is a word in $\Gamma$, 
Properties 1 and 2 again apply to produce a word $Y_1$ in the generators ${\bf a}$, representing the same group element as ${X_1'}$. Furthermore, $|Y_1| \leq |X_1'| \leq 3^{|X_1|}$.

Let $Y_2$ be the word $\varphi(Y_1)$ (where $\varphi$ simply acts individually on each generator).  Then ${\bf u} Y_1 {\bf u}^{-1}Y_2$ is the boundary of a corridor (i.e. a van Kampen diagram consisting of a single corridor) over $\Gamma$, which we call $\mathcal C_{\Gamma}$.  
Since $|Y_2| \leq 3 |Y_1|$ and $|X_1| \leq 3 |X_2|$, we have 
$|Y_2| \leq 3(3^{3|X_2|})$. Clearly $X_1$ and $Y_1$ also satisfy this inequality. 
\end{proof}

\subsection{Construction of $\Delta$} \label{sec:delta}
A van Kampen diagram $\Delta$ for $w$ over $\Gamma$ is obtained 
from $D$ by the following sequence of moves. 
\begin{enumerate}
\item Remove from $D$ all of the open cells of $D \setminus \partial D$ except for open $1$-cells labeled ${\bf u}$ and open $2$-cells whose closures have an edge labeled ${\bf u}$. The result is a circle labeled $w$, with a finite collection of ``open'' ${\bf u}$-corridors attached. Each open ${\bf u}$-corridor is topologically $[0,1] \times (0,1)$, with $0 \times (0,1)$ and $1 \times (0,1)$ identified with open $1$-cells in the circle.  Complete this to get a band complex, i.e.\ a circle with a collection of closed ${\bf u}$-corridor bands attached.  The ${\bf u}$-corridor structure on the closed bands is obtained by pulling back the cell structure and labeling from $D$ as in Section 7.2 of~\cite{bridson}.  
\item Replace each ${\bf u}$-corridor in this band complex with the corresponding corridor over $\Gamma$, guaranteed by Lemma~\ref{lem:corr-dist}.
The result is another band complex, which we denote by $B$.  
\item Remove the open  $1$-cells labeled ${\bf u}$ and the open $2$-cells of $B$ to obtain a disjoint union of circles.  The labels of these circles are called \emph{complementary words}. 

The complementary words $w_i$ are words in  the generators ${\bf a}$ and ${\bf b}$ that represent the trivial element of $\Gamma$.  Let $\Delta_i$ be 
an  area-minimizing van Kampen diagram  over $\Gamma$ for  $w_i$.  
\item Define 
\[
\Delta \; =\;  \left(B \, \sqcup \, \left(\sqcup_{i} \Delta_i \right) \right)/ \sim
\]
where $\sim$ identifies the loop corresponding to $w_i$ in $B$ with $\partial \Delta_i$ for each $i$.
\end{enumerate}

Note that if there are no ${\bf u}$-corridors in $D$, the band complex $B$ is just a circle, and there is just one complementary word $w_1=w$.

\subsection{Upper bound for $\area{}{\Delta}$}
The area of $\Delta$ is simply the sum of the areas of the ${\bf u}$-corridors and the areas of the $\Delta_i$.

We first obtain an upper bound on the total area contribution of the ${\bf u}$-corridors.  Define 
\[ 
L= \begin{cases}\max \;  \{|v| \;|\; v \text{ a horizontal boundary word of a } 
{\bf u} \text{-corridor in } \Delta\} \\ 1\; \text{ if there are no } {\bf u} \text{-corridors} \end{cases}
\]

It follows from inequality~\eqref{eq:corridor-area} that $\area{}{\mathcal{C}} \leq L$ for any ${\bf u}$-corridor $\mathcal{C}$.
Since each ${\bf u}$-corridor intersects the boundary of $\Delta$ in two edges labeled ${\bf u}$, there are at most 
${|w|}/{2}$ such corridors. Thus 
\begin{equation}\label{eq:ucorr-area}
\text{total area of the } {\bf u}\text{-corridors }\leq \frac {|w| L}2.
\end{equation}

Since $\langle {\bf a}, {\bf b} \rangle < \Gamma $ is 
isomorphic to $F_2\times F_2$, and the complementary words $w_i$ are words in ${\bf a}$ and ${\bf b}$, we have 
\begin{equation}\label{eq:comp-area1}
\sum_{i} \area{}{w_i} \leq \sum_{i} (|w_i|)^2
\leq \left( \;\sum_{i} |w_i| \;\right)^2
\end{equation}

From the definition of the $w_i$, we have:
\begin{equation}\label{eq:comp-area2}
\sum_{i} |w_i| \, \leq \, 2(\#\{{\bf u}\text{-corridors}\})L
\,+ \, \left(|w|-2 \#\{{\bf u}\text{-corridors}\}\right) \,
\leq \,|w| L, 
\end{equation}

Putting together the inequalities~(\ref{eq:ucorr-area}),~(\ref{eq:comp-area1}) and~(\ref{eq:comp-area2}), we have
\[
\area{}{\Delta} \;\leq \;\frac{|w|L}2 + (|w|L)^2 \;\leq \;(2|w|L)^2.
\]
The proof of Lemma~\ref{lem:area-distortion} will now follow easily from the above estimate, together with the following two facts. (Just take $\beta_n= 20 \beta$, where $\beta $ is the constant from Fact 2, and recall that $\area{}{D} =  \area{H_n}{w}$.)

\medskip

\noindent {\bf Fact 1.} $|w| \leq 10 \area{}{D}$. 
\begin{proof}
The assumption that $D$ is a topological disk
 implies that each edge of $\partial D$ is part of the boundary of a $2$-cell in
  $D$.  
Since the maximum length of a relation in $H_n$ is $10$, the area of $D$ is at least $\frac{|w|}{10}$.
\end{proof}

\noindent {\bf Fact 2.} There exists a constant $\beta$, independent of $w$, such that  
\begin{equation}\label{eq:logL-squared}
L \leq \beta 3^{ \sqrt{\beta \area{}{D}}}.
\end{equation}
\begin{proof}
There are two cases, depending on the relative sizes of $|w|$ and $L$.

\medskip
\noindent {\it Case (i):} $|w| > (\logl L)^2 $.

By Fact 1 we have 
$L < 3^{\sqrt{|w|}} \leq 3^{\sqrt{10 \area{}{D}}}$, and so inequality~\eqref{eq:logL-squared} holds with $\beta \geq 10$.

\medskip
\noindent {\it Case (ii):} $|w| \leq (\logl L)^2$.

In this case, we establish inequality~\eqref{eq:logL-squared} by obtaining a lower bound on $\area{}{D}$.   More precisely, 
we show that for sufficiently large $L$, $\area{}{D} \geq (\logl L)^2/144$. Recall that $L$ is the maximal length of a horizontal boundary word of a 
${\bf u}$-corridor in $\Delta$.  The idea of the proof is as follows:  since $|\partial \Delta |= |w|$ is relatively small compared to $L$, the existence of a ${\bf u}$-corridor of length $L$ forces $\Delta$ to have a large number of long ${\bf u}$-corridors. This implies that $D$ also has a large number (order $\logl L$) of long (length order $\logl L$) ${\bf u}$-corridors in $D$.  The areas of these corridors of $D$ add up to the required lower bound on $\area{}{D}$. 

We first introduce the notion of the level of a ${\bf u}$-corridor in order to compare lengths of corridors and obtain the above estimates.

\smallskip

\noindent{\bf Levels for ${\bf u}$-corridors.}  
Choose a base corridor $\mathcal C_0$ of $\Delta$ that has a horizontal boundary word of 
maximal length $L$, and define this to be at level $0$.  Now define a corridor $\mathcal C$ of $
\Delta$ to be at level $i$ if 
\begin{enumerate}
\item $\mathcal C$ is not at level $0, \dots, i-1$ and 
\item There exists a path in $\Delta$ connecting $\mathcal C$ to a corridor at level $i-1$ which does not intersect any other ${\bf u}$-corridors.   
\end{enumerate}
Define the level of a ${\bf u}$-corridor in $D$ to be the level of the corresponding ${\bf u}$-corridor in $\Delta$.

Each ${\bf u}$-corridor in $\Delta$ other than $\mathcal C_0$ inherits a notion of top and bottom (with the convention that the horizontal boundary closer to $\mathcal C_0$ is the bottom). For $i>0$ define $\botbound i$ (resp.~$\topbound i$) to be the set of edges which are part of the bottom (resp.~top) boundary of a ${\bf u}$-corridor at level $i$.  Define $\topbound 0$ to be the set of edges which are part of the boundary of $\mathcal C_0$. 

In what follows, $|\topbound i|$ is referred to as the \emph{total ${\bf u}$-corridor length} at level $i$.  Note that for $i>0$,
\[
\frac{|\topbound i|}{3} \leq \text{ total area of } {\bf u}\text{-corridors at level }i \leq |\topbound i|.
\]

Recall that in the construction of $\Delta$ in Section~\ref{sec:delta} we obtain a band complex $B$, basically a circle cut by ${\bf u}$-corridors. Consider a topological circle in $B$ corresponding to a complementary word $w_j$. Assigning to each ${\bf u}$-corridor in $B$ the level it earns in $\Delta$, we see that each such circle passes along the top of a single corridor at level $i-1$ and the bottom of possibly several corridors at level $i$, for some $i$. We say that such a complementary word $w_j$, as well as its van Kampen diagram $\Delta_j$, lies \emph{between} levels $i-1$ and $i$. We define $\deltabound i$ to consist of those edges in $\partial\Delta$ that are also part of a complementary word lying between levels $i-1$ and $i$. Note that $\deltabound i$ includes edges of $\partial\Delta \cap \topbound {i-1}$, and $\partial\Delta \cap \botbound i$.  See Figure~\ref{fig:levels}.

\begin{figure}[ht]
\begin{center}
\psfrag{i}{{\footnotesize Level $i$ corridor}}
\psfrag{m}{{\footnotesize Level $i-1$ corridor}}
\psfrag{c}{{\footnotesize The van Kampen diagram $\Delta_j$}}
\includegraphics[width=120mm]{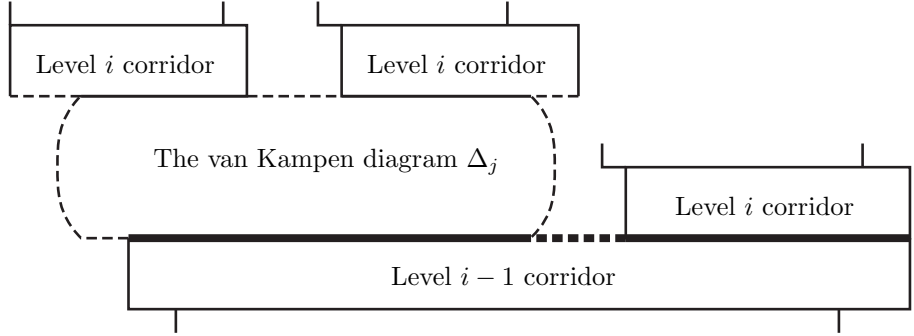}
\end{center}
\caption{A piece of $\Delta$. The dashed lines lie in $\deltabound i$, while the bold line lies in $\topbound {i-1}$. Note that the van Kampen diagram $\Delta_j$ consists of the rounded rectangle along with all the other edges in $\topbound {i-1}$, $\botbound i$, and $\deltabound i$.}
\label{fig:levels}
\end{figure}

\medskip
\noindent {\bf Iterated scaling inequality.}
For $i \geq 0$, consider an edge in $\topbound {i-1}$. Recall that all such edges have labels in ${\bf a}$. If it is not adjacent to a $2$-cell in some $\Delta_j$, then it must belong either to $\botbound i$ (adjacent to a level $i$ corridor) or to $\deltabound i$ (part of $\partial\Delta$).

On the other hand, suppose this ${\bf a}$-edge is part of the boundary of a $2$-cell in some $\Delta_j$. Since every such $2$-cell is labeled by a commuting relation of the form ${\bf a}{\bf b}{\bf a}^{-1}{\bf b}^{-1}$, this $2$-cell is the start of an ${\bf a}$-corridor through $\Delta_j$, with boundary of the form ${\bf a}X_1{\bf a}^{-1}X_2$, where the $X_i$ are words in ${\bf b}$.

Assuming the ${\bf a}$ term in the expression above corresponds to the original edge in $\topbound {i-1}$, we claim that the edge corresponding to the ${\bf a}^{-1}$ term must lie either in $\botbound i$ or $\deltabound i$.

To see this, note that from the description of the complementary words above, it is clear that the only other option is that this edge lies in the same component of $\topbound {i-1}$, along the same ${\bf u}$-corridor (forming an arch in Figure~\ref{fig:levels}). Suppose for contradiction that this is the case. Then this ${\bf a}$-corridor forms, along with a portion of the original ${\bf u}$-corridor, an annulus of $2$-cells in $\Delta$, the inner boundary of which corresponds to a product of a nontrivial word in ${\bf a}$ with a nontrivial word in ${\bf b}$. As no such product can be trivial $\Gamma$, we deduce that any ${\bf a}$-corridor with one boundary ${\bf a}$-edge in $\topbound {i-1}$ has its other boundary ${\bf a}$-edge in either $\botbound i$ or $\deltabound i$.

It follows from this analysis that $|\topbound {i-1}| \leq |\botbound i| +| \deltabound i|$, for $i\geq1$. Since conjugation by ${\bf u}^{\pm1}$ scales ${\bf a}$-words by at most a factor of $3$, we have 
$|\botbound i| \leq 3|\topbound i|$.
Combining these two inequalities, we have 
\[
|\topbound {i-1}| \leq 3 |\topbound {i} | + | \deltabound i|.
\]
By iterating this and noting that $L < |\topbound 0|$,
we obtain:
\[
L < |\topbound 0| \leq 3^i |\topbound i| + \sum_{j=1}^{i}3^{j-1} |\deltabound j|, \quad \text{ for } i>0. 
\]
Since $\sum_{j=1}^{i} |\deltabound j| \leq |w|$ for each $i$, this implies
\begin{equation}\label{eq:Li}
L  < 3^i |\topbound i| + 3^i |w|.
\end{equation}

\medskip
\noindent {\bf Large $L$ implies many long levels.}
Now we show that if $L$ is sufficiently large, then there are at least $\frac{\logl L}4$ levels in $\Delta$, each with total ${\bf u}$-corridor length at least $\sqrt L$.

Since $ \lim_{L \to \infty} \frac{(\logl L)^2}{\sqrt{L}} = 0$, there exists $P>0$ such that $(\logl L)^2 \leq \sqrt{L}$ for all $L\geq P$. 
Note that $P$ depends only on the functions $(\logl x)^2$ and $\sqrt x$ and not on $w$. 
For $L \geq P$, 
inequality~\eqref{eq:Li} and the base inequality of Case (ii) give
\[
L \leq 3^i|\topbound i| + 3^i |w| \leq 3^i|\topbound i| + 3^i (\logl L)^2\leq 3^i |\topbound i| +3^i \sqrt{L}.
\]

Rearranging gives $|\topbound i| \geq  3^{-i}L -\sqrt{L}$. 
The reader can now verify that if $i \leq \frac{\logl L}4$ and $L \geq 16$, then $3^{-i}L \geq 2 \sqrt L$.  As a consequence, we have: $$|\topbound i|\geq \sqrt L.$$

In summary, if $L\geq \max\{ P, 16\}$, then there are at least $\frac{\logl L}4$ levels in $\Delta$, each with total ${\bf u}$-corridor length at least $\sqrt L$.

\medskip
\noindent {\bf Relating $\area{}{D}$ and $(\logl L)^2$ for large $L$.}
Now we can estimate the area of $D$ from below using ${\bf u}$-corridors.
\begin{align*}
\area{}{D} & \geq \text{total area of } u\text{-corridors in } D\\
& \geq \text{total area of } u\text{-corridors in } D \text{ at or below level } \frac{\logl L}{4} \\
& \geq \left( \frac{\logl L}{4} \right) \left(\frac{\logl (\sqrt L/3)}{9} \right) \\
&= \frac{(\logl L)^2}{72} - \frac{\logl L}{36}\\
&\geq \frac{(\logl L)^2}{144} \qquad \text{ provided } L \geq 3^4.
\end{align*}
The second term in the third inequality above is obtained as follows. 
Note that $\sqrt L$ is a lower bound for the total ${\bf u}$-corridor length at level $i$ in $\Delta$, and hence is a lower bound for a sum $\sum |Y_j|$, where $j$ runs over an index set for all ${\bf u}$-corridors at level $i$, and the $Y_j$ are horizontal boundary words of these corridors. Thus, by inequality~\eqref{eq:length-comp},
\[
\sqrt L\; \leq \;\sum |Y_j|\; \leq \;\sum 3\,(3^{3|X_j|}) \;\leq \; 3\,(3^{3{\sum |X_j|}}),
\]
and so the total ${\bf u}$-corridor length at level $i$ in $D$ (which is ${\sum |X_j|}$ in the inequality above) is at least
$\logl (\sqrt L/3)/3$.  Finally, by inequality~\eqref{eq:corridor-area}, the total area of ${\bf u}$-corridors at level $i$ in $D$ is at least 
$\logl (\sqrt L/3)/9$.

\medskip
\noindent{\bf Summary.} We have shown that in Case (i), we have $L \leq 3^{\sqrt{10 \area{}{D}}}$, and in Case (ii), we have $L \leq 3^{\sqrt{144 \area{}{D}}}$, provided 
$L \geq \max\{ 3^4, 16, P\}$.  Thus inequality~\eqref{eq:logL-squared} holds for all $L$, provided we take $\beta = \max\{ 10, 144, 3^4, 16, P\} =  \max \{144, P\}$.  Since $P$ was independent of $w$, so is $\beta$. 
\end{proof}

\begin{remark}
The notion of area distortion as a group invariant is defined in Section 2 of~\cite{gersten-area-distortion}.  Lemma~\ref{lem:area-distortion} provides an upper bound for the area distortion of 
$\Gamma$ in $H_n$.   The reader can verify that the boundary words for the van Kampen diagrams $\ddi{i}{n}{N}$ from Section ~\ref{sec:lower} establish the lower bounds for this distortion.  Thus the distortion of $\Gamma$ in $H_n$ is $f(x)\simeq e^{\sqrt x}$.  It would be interesting to find other pairs $(G,H)$ of 
type $({\mathcal F}_{3}, {\mathcal F}_{2})$ 
where area distortion can be explicitly computed.  
For example, the subgroup of the group $H_n$  which is  generated by $\{\stable{k}{i} \, | \, k \leq n, 
i = 1,2\}$ should have area distortion $\exp^n\sqrt{x}$. 
\end{remark}

\end{sloppypar}

\bibliographystyle{siam}
\bibliography{refs-dehn}

\end{document}